\documentclass[leqno,12pt]{article} 

\usepackage{authblk}

\usepackage{fullpage}

\usepackage{color}
\usepackage{amssymb}
\usepackage{amsfonts}
\usepackage{amscd}
\usepackage{amstext}
\usepackage{latexsym,amssymb,amsthm}
\usepackage{indentfirst}
\usepackage{amsmath}
\usepackage{float}
\usepackage[numbers]{natbib}
\usepackage[utf8]{inputenc}
\usepackage{ytableau}
\usepackage[vcentermath]{youngtab}

\usepackage{setspace}

\newtheorem{teo}{Theorem}[section]

\newtheorem{lema}[teo]{Lemma}

\newtheorem{propo}[teo]{Proposition}

\newtheorem{defi}[teo]{Definition}

\newtheorem{obs}[teo]{Remark}

\newtheorem{coro}[teo]{Corollary}

\newtheorem{exem}[teo]{Example}

\newcommand{\cn}{\langle n \rangle}

\newcommand{\dis}{\displaystyle}

\author[1]{\textbf{Irina Sviridova}}
\author[2]{\textbf{Renata A. Silva}\thanks{Both authors are partially supported by CAPES and CNPq.}}
\affil[1]{\small Mathematics Department -- University of Brasília \\ 70.910-900 -- Brasília, DF, Brazil \\  Email: \texttt{i.sviridova@mat.unb.br}}
\affil[2]{\small Mathematics Department -- Federal University of Northern Tocantins\\ \small 77.826-612 -- Araguaína, TO, Brazil \\ \small Email: \texttt{renata.silva@ufnt.edu.br}}

\date{}

\title{Hook Theorem for Superalgebras with Superinvolution or Graded Involution}

\begin{document}

\maketitle

\begin{abstract} 

\noindent We consider a superalgebra with a superinvolution or graded involution $\#$ over a field $F$ of characteristic zero and assume that it is a $PI$-algebra. In this paper, we present the proof of a version of the celebrated hook theorem \cite{SAR} for the case of multilinear $\#$-superidentities. This theorem provides important combinatorial characteristics of identities in the language of symmetric group representations. Furthermore, we present an analogue of Amitsur identities for $\#$-superalgebras, which are polynomial interpretations of the mentioned combinatorial characteristics, as a consequence of the hook theorem.

\end{abstract}

\noindent\textbf{Mathematical Subject Classification MSC2020: 16R10, 16R50, 16W50, 16W55, 16W10.}

\vspace{0.5cm}

\section{Introduction}
\onehalfspacing

Let $A$ be an associative algebra over a field $F$ of characteristic zero. Let $F\langle X \rangle$ be the associative, non-commutative and non-unitary free algebra  generated by an infinite countable set $X=\{x_1,x_2,\cdots\}$. We say that a polynomial $f(x_1, \cdots, x_n) \in F\langle X \rangle$ is a polynomial identity of $A$ if $f(a_1, \cdots, a_n)= 0$ for all $a_1, \cdots, a_n \in A$. We call $A$ $PI$-algebra if $A$ satisfies a non-trivial polynomial identity. For example, Amitsur and Levitsky \cite{AL} proved that the standard polynomial of degree $2k$, given by $St_{2k}(x_1, \cdots, x_ {2k})= \displaystyle \sum_{\sigma\in S_{2k}}(-1)^\sigma x_{\sigma(1)}x_{\sigma(2)}\cdots x_{\sigma(2k )}$, is a polynomial identity of the smallest possible degree (minimum degree) for the algebra of matrices of order $k$ over $F$. Other principal results of $PI$-algebra theory may be found in \cite{DR,DrenForm,RV}.

We denote by $Id(A)$ the set of ordinary polynomial identities of $A$. The set $Id(A)$ is a $T$-ideal, that is, an ideal of $F\langle X \rangle$ invariant under all endomorphisms of $F\langle X \rangle$. Thus, to describe the identities of $A$, it is enough to find the generators of $Id(A)$ as a $T$-ideal. Since the description of a base of identities of $Id(A)$ is not always an easy task, in order to overcome this difficulty, in 1972, Regev \cite{AR} introduced the sequence of codimensions (ordinary case), which measures the growth of the $T$-ideal of identities for $A$. The $n$-th codimension of $A$, $c_n(A)$, is defined as the dimension of the quotient space $P_n(A)=\displaystyle \frac{P_n}{P_n \cap Id(A)}$ , where $P_n$ is the space of the multilinear polynomials of degree $n$. The importance of considering spaces $P_n$ in $PI$-theory lies in the fact that, over a field of characteristic zero, every polynomial identity of $A$ is equivalent to a finite set of multilinear polynomials.

In \cite{AR}, Regev also showed that the sequence of codimensions of a $PI$-algebra $A$, $(c_n(A))_{n\geq 1}$, is exponentially bounded (while $\dim P_n=n!$), that is, there is a positive integer $ d $ such that, for every $ n \geq 1 $, $c_n (A) \leq d^n $. As a consequence of Regev Theorem for ordinary identities, it can be shown that, if $ A $ has an additional structure, for example, of graded algebra, superalgebra with superinvolution or graded involution $ \# $ (see \cite{ARA}), then its corresponding sequence of codimension $ (c_n^{grs} (A))_{n \geq 1} $ is also exponentially bounded. 

Over the years, this important result of Regev served to obtain several other theorems in $PI$-algebra. In particular, Amitsur and Regev, in \cite{SAR} (see also \cite{RV}), proved the hook theorem for the case of ordinary polynomial identities of $ A $, which is one of the most essential combinatorial results for polynomial identities.

The hook theorem is the central result of applying group representation theory to the study of polynomial identities. This theorem is a crucial foundation for many other significant results in the study of polynomial identities. One of the main applications of the hook theorem are the crucial points of the proof of the well-known Kemer's Theorem \cite{Kem1}, which guarantees that any $PI$-algebra, over a field of characteristic zero, has the same identities as the Grassmann envelope $E(A)$ of a superalgebra $ A$ of a finite dimension. The relevance of this result is given by responding positively to Specht's conjecture, proving that any $T$-ideal of identities of a $PI$-algebra over a field of characteristic zero is finitely generated. We highlight that in 2010, Aljadev and Belov proved Kemer's Theorem for $G$-graded algebras, where $G$ is a finite group (see \cite{AljBelov}). In 2011, Sviridova demonstrated a similar result for superalgebras (see \cite{IRINA}). Later, in 2016, Aljadev, Giambruno, and Karasik showed in \cite{alj1} that Kemer's Theorem also applies to algebras with involution $*$.

The hook and strip theorems have also played a fundamental role in the in the creation and development of the renowned growth theory of polynomial identities, a field that has attracted significant attention from researchers in recent years. This area investigates the behavior of combinatorial characteristics of identities, such as codimensions, colengths, multiplicities in cocharacters, exponents, and others. Some of the key results of the last years can be found in \cite{A.B.1,gmv1,RV,ioio3}. We also mention some recent results of this area, related with graded identities, identities with involution and superidentities with superinvolution \cite{CG1,IOIO,AV2}. Observe that this area has experienced significant development in recent years, with numerous published papers and many researchers actively contributed to the field. While it is impossible to mention everyone, we apologize for any omission of important results or key contributors to the growth theory.

A version of the hook theorem, considering $A$ an algebra with involution or a superalgebra, was demonstrated by Regev and Giambruno in \cite{RG}. The authors used the theory of representation of the wreath product $\hat{\mathbb{Z}}_2=\mathbb{Z}_2\wr S_n$ of the group $\mathbb{Z}_2$ by the symmetric group $S_n$, applying it to the study of $\mathbb{Z}_2$-identities (superidentities) or $*$-identities (identities with  involution $*$) of $A$, that is, identities with $\hat{ \mathbb{Z}}_2$ action, and proved that the corresponding $n$-th cocharacter
$$
\chi_{n}(A| \mathbb{Z}_2)= \displaystyle \sum_{|\lambda| + |\mu|=n}m_{\lambda, \mu} \chi_{\lambda, \mu} \ \ \text{(for superidentities)}
$$
or
$$
\chi_{n}(A| *)= \displaystyle \sum_{|\lambda| + |\mu|=n}m_{\lambda, \mu} \chi_{\lambda, \mu} \ \ \text{(for} \ * \text{-identities)}
$$
is contained in a double hook $(\mathcal{H}(d_1, l_1), \mathcal{H}(d_2, l_2))$, with $d_1, l_1, d_2,l_2$ nonnegative integers.

Let consider an associative superalgebra $A=A_0\oplus A_1$, over a field of characteristic zero, with a graded involution or a superinvolution ($\#$-superalgebra). 

Note that superidentities with superinvolution and superidentities with graded involution behave mostly in similar ways and can be considered in the same manner. We call them $\#$-superidentities. In general, they can be considered as polynomials in 4 different types of variables: even symmetric $Y_0$, even skew $Z_0$, odd symmetric $Y_1$, odd skew $Z_1$, and form the representation space of the direct product of 4 copies of symmetric groups $S_{n_1} \times S_{n_2} \times S_{n_3} \times S_{n_4}$ acting independently in any type of these variables.

Let $n \geq 1$ be an integer, where $n=n_1 + n_2 + n_3+ n_4$ a sum of four nonnegative integers, and define   $\cn= (n_1, n_2, n_3, n_4)$. The first main contribution of this article is to prove the following result.

\noindent \textbf{Theorem \ref{teoprin}:} \textit{(The hook theorem for $\#$-superalgebras) Let $F$ be a field of characteristic zero and $A$ a $\#$-superalgebra over $F$. If $A$ is a $PI$-algebra, then there exist integers $d_i, l_i\geq 0$, with $i= 1,2,3, 4$, such that the $\cn$-th cocharacter of $\#$-superidentities, $\chi_{\cn}(A)$, is contained in a quadruple hook $\mathcal{H}_{\langle d, l \rangle}\langle n \rangle= (\mathcal{H}(d_1, l_1),\mathcal{H}(d_2, l_2), \mathcal{H}(d_3, l_3), \mathcal{H}(d_4, l_4) )$, that is, 
$$
\chi_{\langle n\rangle}(A)= \displaystyle\sum_{\overset{
\langle \lambda \rangle\vdash \cn}
{\langle \lambda \rangle \in \mathcal{H}_{\langle d, l \rangle}\langle n \rangle}}
m_{\langle \lambda \rangle}\chi_{\langle \lambda \rangle}.
$$}

As a consequence of Theorem \ref{teoprin}, we have a version of Amitsur's Theorem for $\#$-superalgebras represented in Corollary \ref{amitsur1}, which is mostly some combinatorial improvement of numerical parameters of the corresponding $\#$-superidentities of the form $St_{k_i}^{m_i} \equiv 0$ comparing it with the classic result.

Amitsur and Regev also proved in \cite{SAR} (see also Theorem 4.7.2 in \cite{RV}) that any $PI$-algebra satisfies some Amitsur identity, and this identity is an interpretation of the hook theorem on the language of polynomials. This correspondence is a pure combinatorial fact. It is worth mentioning as well that the Amitsur identity is the natural generalization of the famous Capelli identity (\cite{DR,DrenForm,RV,Kem1}).
Thus, it is a natural problem to find a polynomial translation of Theorem \ref{teoprin} for $\#$-superidentities. In particular, A. Berele asked the authors
in a discussion if there exist some analogues of Amitsur polynomials for $\#$-superidentities. We are grateful to A. Berele for this useful observation and give a positive answer to this question in the last section of the paper.

We define polynomials $E_{\langle d,l\rangle},$ which are some analogues for $\#$-superidentities of the classic Amitsur polynomial $Am_{(d,l)}$ (see Definition \ref{CapSup1}) and prove Theorem \ref{AmPol}.

\noindent \textbf{Theorem  \ref{AmPol}:}
\textit{Let $A$ be a superalgebra with superinvolution or graded involution $\#$ and
$
\chi_{\cn}(A) = \dis\sum_{\langle\lambda\rangle\vdash \cn}m_{\langle\lambda\rangle}\chi_{\langle\lambda\rangle}
$
its $\cn$-th $\#$-cocharacter. If $A$ is a $PI$-algebra, then $A$ satisfies $E_{\langle d,l\rangle}$ (the Amitsur $\#$-superidentity of rank $\langle d,l\rangle$) for some $\langle d,l\rangle=(d_1,l_1;d_1,l_2;d_3,l_3;d_4,l_4)$ if, and only if, $m_{\langle\lambda\rangle} = 0$ whenever $\langle\lambda\rangle \notin \mathcal{H}_{\langle d,l\rangle}$.}

Theorems \ref{teoprin} and \ref{AmPol} jointly imply that any $\#$-$PI$-superalgebra over a field of characteristic zero satisfies an Amitsur $\#$-superidentity $E_{\langle d,l\rangle}.$  It completes the description of $\#$-superidentities in the language of shapes of Young diagrams.

\section{Superalgebras, Superinvolutions, Graded Involutions, Identities.}

Throughout this work, all algebras will be considered as associative algebras over a field $F$ of characteristic zero.

Let $A = A_0 \oplus A_1$ be an associative superalgebra over a field $F$. A superinvolution in $ A $ is a graded linear application $\#: A \rightarrow A$ such that $ (c^\#)^\# = c $ for all $c \in A$ and $(ab)^{\#} = (-1)^{deg (a) deg (b)} b^{\#} a^{\#}$ for every homogeneous elements $ a, b \in A_0 \cup A_1 $, where $ deg (d) $ is the homogeneous degree of $ d \in A_ {0} \cup A_ {1} $. In particular, there hold $ A_0^{\#} \subseteq A_0 $ and $ A_1^{\#} \subseteq A_1 $. 

A graded involution in $A$ is a linear application $*: A \rightarrow A$ such that 
$(c^*)^* = c$ for all $c \in A$ and $ (ab)^{*} = b^{*} a^{*} $ for all elements $ a, b \in A$. In addition, $*$ is a graded map, that is, $A_0^{*} \subseteq A_0$ and $A_1^{*} \subseteq A_1$. Note that the unique difference between a superinvolution $\#$ and a graded involution $*$ of a superalgebra $A$ is the condition $(ab)^{\#} = - b^{\#} a^{\#}$ for all odd elements $a, b \in A_1$ for a superinvolution $\#$, while $(ab)^{*} = b^{*} a^{*}$ holds for all odd elements $a, b \in A_1$ for a graded involution $*$.

By this similarity, we will treat a superinvolution or a graded involution in the same way and denote both applications by the same symbol $\#$ throughout the paper. When necessary, differences for each of these mappings will be mentioned. 

If $\#$ is a superinvolution or a graded involution of a superalgebra $A$, then we say that $A$ is a $\#$-superalgebra. Hereafter, unless otherwise specified, we will consider $A$ exclusively as a $\#$-superalgebra.

Since $ char (F) \neq 2 $, it can be shown that $ A = A^{+} \oplus A ^ {-} $, where $A^{+} = \{a \in A \ | \ a^{\#} = a \} $ and $ A^{-} = \{a \in A \ | \ a^{\#} = -a \} $ the symmetric and skew sets of $ A $, respectively. Thus, since $ A $ is a superalgebra, it can be written as a sum of  four subspaces
\begin{equation} \label{A-decomp}
A = A_0^{+} \oplus A_0^{-} \oplus A_1^{+} \oplus A_1^{-},
\end{equation}
where $ (A_i)^+ = \{a \in A_i \ | \ a^\# = a \} $ and $ (A_i)^- = \{a \in A_i \ | \ a^\# = - a \} $ for $ i = 0 $ and 1.

Consider $\mathcal{F} = F\langle X | \mathbb {Z}_2, \# \rangle $ the free algebra over a field $F$ of characteristic zero, generated by a countable set $X$. We write the set $X$ as a disjoint union of four countable sets $ X = Y_0 \cup Z_0 \cup Y_1 \cup Z_1 $, where $ Y_0 = \{y_ {0,1}, y_ {0,2}, \cdots \}, \ Z_0 = \{z_ {0,1}, z_ {0,2}, \cdots \}, \ Y_1 = \{y_ {1,1}, y_ {1,2}, \cdots \}, \ Z_1 = \{z_ {1,1}, z_ {1,2}, \cdots \} $ are sets of variables that are even symmetric, even skew, odd symmetric and odd skew, respectively.

Note that $\mathcal{F}$ can be equipped with a superstructure assuming that the variables of $ Y_0 \cup Z_0 $ and $ Y_1 \cup Z_1 $ are homogeneous of degree $ 0 $ and $ 1 $, respectively. Therefore, $ \mathcal{F} = \mathcal{F}_0 \oplus \mathcal{F}_1 $, where $ \mathcal{F}_0 $ is the subspace generated by all monomials that have an even number of variables of degree $ 1 $ and $ \mathcal{F}_1 $ is the subspace generated by all monomials that have an odd number of variables of degree $ 1 $. 

In addition to this superstructure, we can define a superinvolution or a graded involution in the free superalgebra $\mathcal{F}$. To do this, let $\#:\mathcal{F} \to \mathcal{F}$ be the map defined by $y_{i,j}^\#=y_{i,j}$, $z_{i,j}^\#=-z_{i,j}$, for $i=0,1$ and $j\geq 1$, and either
\begin{equation} \label{supinv}
w^\#=(x_{i_1}x_{i_2}\cdots x_{i_k})^\# = (-1)^{\frac{s(s-1)}{2}} x_{i_k}^\#x_{i_{k-1}}^\#\cdots x_{i_1}^\#= (-1)^{\frac{s(s-1)}{2}} (-1)^{t} x_{i_k}x_{i_{k-1}}\cdots x_{i_1}
\end{equation}
for a superinvolution $\#,$ or
\begin{equation} \label{grinv}
w^\#=(x_{i_1}x_{i_2}\cdots x_{i_k})^\# = x_{i_k}^\#x_{i_{k-1}}^\#\cdots x_{i_1}^\#= (-1)^{t}x_{i_k}x_{i_{k-1}}\cdots x_{i_1}
\end{equation}
for a graded involution $\#,$ 
where $w$ is a monomial in the indeterminates $x_{i_l}\in Y_0\cup Z_0\cup Y_1\cup Z_1$, $s=deg_{Y_1\cup Z_1}(w)$ is the number of odd indeterminates in the monomial $w$, and $t= deg_{Z_0\cup Z_1}(w)$ is the number of skew indeterminates in the monomial $w$. Thus, the variables of $ Y_0 \cup Y_1 $  are symmetric with respect to $\#$ and the variables of $ Z_0 \cup Z_1 $ are skew.

Now, extend the map $\#$ by linearity to the entire algebra $\mathcal{F}$. We see that $\#$ is, in fact, a superinvolution on $\mathcal{F}$ in the first case (\ref{supinv}), and a graded involution in the second case (\ref{grinv}). Thus, $ \mathcal{F} $ has a structure of $\#$-superalgebra, since $ \mathcal{F}_0^{\#} = \mathcal{F}_0 $ and $\mathcal{F}_1^{\#} = \mathcal{F}_1$. 
We call by $\#$-superpolynomials the elements of $\mathcal{F}$.

Let
 $$f= f(y_{0,1}, \cdots, y_{0,m}, z_{0,1}, \cdots, z_{0,n},y_{1,1}, \cdots, y_{1,p},z_{1,1}, \cdots, z_{1,q}  )\in F\langle X| \mathbb{Z}_2,\#\rangle$$
be a nonzero $\#$-polynomial. We say that $f$ is a $\#$-superidentity (a superidentity with superinvolution or superidentity with graded involution) of a  $\#$-superalgebra $A$, and we write $f\equiv 0$ in $A$, if 
\begin{equation}\label{eqsuper1}
f= f(a_{0,1}, \cdots, a_{0,m}, b_{0,1}, \cdots, b_{0,n},a_{1,1}, \cdots, a_{1,p},b_{1,1}, \cdots, b_{1,q}  )=0
\end{equation}
for all $a_{0,1}, \cdots, a_{0,m}\in A_0^{+},\ b_{0,1}, \cdots, b_{0,n} \in A_0^{-},\ a_{1,1}, \cdots, a_{1,p}\in A_1^{+},\ b_{1,1}, \cdots, b_{1,q} \in A_1^{-} $. We define
$$
Id_{2}^{\#}(A):=\{f\in F\langle X| \mathbb{Z}_2,\#\rangle \ | \ f\equiv 0\ \text{in}\ A\}
$$
the  $T_{2}^{\#}$-ideal of $\#$-superidentities of $A$, that is, a graded $\#$-invariant ideal of $\mathcal{F}$, closed under all graded endomorphisms of $F\langle X| \mathbb{Z}_2,\#\rangle$ that commute with $\#$. Since we are assuming that  $char (F)=0$, due to the multilinearization process, the ideal $ Id_{2}^{\#}(A)$ is completely determined by its multilinear polynomials, and so we define
$$P_n^{grs}:= \text{span}_{F}\{w_{\sigma(1)}\cdots w_{\sigma(n)}|\ \sigma\in S_n, w_i=y_{j,i}\ \text{or}\ w_i=z_{j,i},\ j=0,1,\ i=1,\cdots, n \}$$
the space of multilinear $\#$-superpolynomials of degree $n$ on the first $n$ variables.

\begin{defi} The non-negative number 
$$c_n^{grs}(A)= dim \frac{P_n^{grs}}{P_n^{grs}\cap Id_{2}^{\#}(A)} $$
is called $n$-th $\#$-graded codimension of $A$.
\end{defi}

In [\citealp{AR}], A. Regev showed that, if $ A $ is an ordinary $ PI $-algebra, then the sequence of codimensions of ordinary identities  of $ A $ is exponentially bounded (see also [\citealp{RV}]), that is, there is a real number $ d \geq 1 $ such that, for every $ n \geq 1 $,
\begin{equation} \label{REG1}
c_n (A) \leq d^n.
\end{equation}

The following result relates the $n$-th $\#$-graded codimension with the ordinary codimension of $A$. In \cite[p. 257]{RV}, the authors proved a similar result for the case of a $PI$-algebra satisfying a nontrivial $G$-identity, where $G$ is a finite group of $Aut^*(A)$, the group of automorphisms and antiautomorphisms of $A$. We will present the proof here for the sake of completeness.

\begin{lema}  \label{Lem 1}
 Let $ A $ be a $ \# $-superalgebra. Then, for any $ n \geq 1 $, we have
$$c_n^{grs}(A)\leq 4^nc_n(A).$$
\end{lema}

\begin{proof} Let $\{f_1(x_1,\cdots,x_n),\cdots,f_k(x_1,\cdots,x_n)\}$ be a basis of the space $P_{n}$ of all ordinary multilinear polynomials of degree $n$ modulo $P_n\cap Id(A)$, where $k=c_n(A)$. Then, for every $\sigma\in S_n$, we can write the monomials of $P_n$ as
$$
x_{\sigma(1)}\cdots x_{\sigma(n)}\equiv \displaystyle \sum_{i=1}^k \alpha_{i,\sigma}f_i(x_1,\cdots,x_n)
$$
modulo $P_n\cap Id(A)$. Now, consider the following notation. Given a variable $x_i$, $i\in\{1,\cdots,n\}$, we will substitute $x_i={w_i}$ with a variable in the set $Y_0\cup Z_0\cup Y_1\cup Z_1$, where $w_i\in \{y_{0,i}, z_{0,i}, y_{1,i}, z_{1,i}\}$ for all $i=1,\cdots, n$, in all possible ways. Thus,
$$
w_{\sigma(1)}\cdots w_{\sigma(n)} \equiv \displaystyle \sum_{i=1}^k \alpha_{i,\sigma}f_i(w_1,\cdots,w_n)
$$
modulo $Id_2^\#(A)$, since $Id(A)\subseteq Id_2^\#(A)$. Hence, the set
$$
\{f_i(w_1,\cdots,w_n); \ w_j\in \{y_{0,j}, z_{0,j}, y_{1,j}, z_{1,j}\}, \ 1\leq i\leq k,\ \ 1\leq j\leq n\}
$$
generates the space $P_n^{grs}$ modulo $Id_2^\#(A)$ and, therefore,
$$
c_n^{grs}(A)\leq 4^nc_n(A),
$$
as we wanted.

\end{proof}
Thus, from Lemma \ref{Lem 1} and from (\ref{REG1}), the following corollary follows immediately.

\begin{coro} \label{corregv}
Let $ A $ be a $ \# $-superalgebra. If $ A $ is a $ PI $-algebra, then its sequence of $ \# $-graded codimension $ (c_n^{grs} (A))_{n \geq 1} $ is exponentially bounded, that is, there is a real number $ \overline {d} $ satisfying for every $ n \geq 1 $
$$
c_n^{grs} (A) \leq (\overline {d})^n.
$$
\end{coro}
 
\section{$FS_{\cn}$-representations and Action of $S_{\cn}$ in $P_{\cn}$}

We begin this subsection recalling some concepts from the well-known representation theory of the symmetric group $S_n$ over a field $F$ of characteristic zero, utilizing Young tableaux, where $n \geq 1$.

\begin{defi} Let $n \geq 1$ be an integer. A \emph{partition} $\lambda$ of $n$ is a sequence of natural numbers $\lambda = (\lambda_1, \cdots, \lambda_l)$ such that $\lambda_1 \geq \cdots \geq \lambda_l$ and $\sum_{i=1}^l \lambda_i = n$. In this case, we write $\lambda \vdash n$ and $|\lambda| = n$. Also, given a partition $\lambda = (\lambda_1, \cdots, \lambda_l)$ of $n$, a \emph{Young diagram} of shape $\lambda$, denoted by $D_\lambda$, is a matrix or array of $n$ boxes arranged in rows such that the $i$-th row contains $\lambda_i$ boxes. The boxes are left-justified and, by convention on partitions, the lengths of the rows are non-increasing. The number $h(\lambda)=l$ is called the height of $D_\lambda$ (or of $\lambda$).
\end{defi}

We can represent a Young diagram by the set
$$
D_{\lambda} = \{ (i, j) \in \mathbb{Z} \times \mathbb{Z} \mid i = 1, \cdots, l;\ j = 1, \cdots, {\lambda}_i \}.
$$
For example, the diagram $D_{(5,3,2,1)}$, corresponding to the partition of $n=11$, is represented by
$$
D_{(5,3,2,1)} = \yng(5,3,2,1)
$$
and its height is equal to $h((5,3,2,1))=4$. 

We assume $D_{\lambda}=\emptyset$ is the empty diagram of the zero partition $\lambda=(0)$ for $n=0$, where the correspond symmetric group $S_0$ is the trivial group.

\begin{defi} A \emph{Young tableau} of shape $\lambda$ is a bijective assignment of the integers ${1, \cdots, n}$ to the $n$ boxes of the Young diagram. We denote a Young tableau of shape $\lambda$ by $T_{\lambda} = D_{\lambda}(a_{ij})$, where $a_{ij}$ is the assignment to the box that is in the $i$-th row and $j$-th column.
\end{defi}

For example, a Young tableau of shape $\lambda=(5,3,2,1)$, a partition of $n=11$, is the following:
$$ 
T_{(5,3,2,1)}= \ytableausetup{mathmode, boxsize=1.2em,centertableaux}
\begin{ytableau}
1 & 2 & 3 & 4 & 5\\
6 & 7 & 8 \\
9 & 10 \\
11
\end{ytableau}.
$$

Let $n\geq 1$ be an integer. It is possible to define an action of the symmetric group $S_n$ in the set of Young tableaux in the following way: $S_n$ acts on a tableau $T_{\lambda}$ by permuting each entry individually. That is, if $\sigma \in S_n$ and $T_{\lambda} = D_{\lambda}(a_{ij})$, then
$$
\sigma T_{\lambda} = D_{\lambda}(\sigma a_{ij}).
$$

This action turns the space of Young tableaux into an $FS_n$-module and one of the main tools for describing the representations of the symmetric group. The reader can consult the references \cite{CR,jlieb,ben1} for the general theory of group representations and, for the specific case of group $S_n$, Young diagram and tableau applications, the books \cite{fulton, RV, jkerb, sagan}.

Now, we will generalize these concepts and some results of $S_n$-representation theory for the direct product of four symmetric groups.

For an integer $ n\geq 1 $, write $ n = n_1 + n_2 + n_3 + n_4 $, where each $ n_i\geq 0 $, and denote $ \langle n \rangle = (n_1, n_2, n_3, n_4 ) $. We say that $ \langle \lambda \rangle = (\lambda (1), \lambda (2), \lambda (3), \lambda (4)) $ is a multipartition of $\cn$ and we write $ \langle \lambda \rangle \vdash \langle n\rangle $ if
$ \lambda (i) \vdash n_i $ is a partition of $ n_i $, $i=1,2,3,4$. Associated to a multipartition $ \langle \lambda \rangle $, 
we define $T_{\langle\lambda\rangle} = (T_{\lambda(1)}, T_{\lambda (2)}, T_{\lambda(3)}, T_{\lambda (4)}) $ its Young multitableau of shape $ \langle \lambda \rangle $, where each $ T_{\lambda (i)} $ is a Young tableau corresponding to the partition $ \lambda (i) $.

We will give here some important results on the representation theory of the group $S_{\langle n \rangle}:= S_{n_1}\times S_{n_2}\times S_{n_3}\times S_{n_4}=\{\langle\sigma\rangle = (\sigma_1,\sigma_2,\sigma_3,\sigma_4) \ | \ \sigma_i\in S_{n_i}, \ i=1,2,3,4\}$, where we denote by $S_0=S_1$ the trivial group (for $n_i=0$ or $n_i=1$).

We consider left $FS_{\langle n\rangle}$-modules, where $F$ is a field of characteristic zero. It is well known that in this case $F S_{\langle n \rangle}$ is a semisimple finite dimensional associative algebra (see, e.g., \cite{CR, RV, jacobs}) and any
left $FS_{\langle n\rangle}$-module is completely reducible (\cite{CR, jlieb}). In addition, any irreducible $F S_{\langle n \rangle}$-character corresponds to a multipartition $ \langle \lambda \rangle$ and is equal to the tensor product of the irreducible characters of $ S_{n_1}, \cdots, S_{n_4 } $ corresponding to $\lambda (1), \lambda (2), \lambda (3), \lambda (4)$, respectively. 
We denote by $ \chi_{\langle \lambda \rangle} = \chi_{\lambda (1)} \otimes \chi_{\lambda (2)} \otimes \chi_{\lambda (3)} \otimes \chi_{ \lambda (4)} $ 
the $FS_{\cn} $-irreducible character associated to the multipartition $\langle \lambda \rangle $ and $ d_{\langle \lambda \rangle} = d_{\lambda (1)} \cdot d_{\lambda (2)} \cdot d_{\lambda (3)} \cdot d_{\lambda (4)} $ its degree, where $ d_{\lambda (i)} $ is the degree of the irreducible character $ \chi_{\lambda (i)} $, $ i = 1,2,3,4 $. We consider the decomposition of the character of any left $FS_{\langle n\rangle}$-module into irreducible characters as follows:
\begin{equation}\label{caracter}
\chi_{\cn} = \dis\sum_{\langle\lambda\rangle\vdash \cn}m_{\langle\lambda\rangle}\chi_{\langle\lambda\rangle},
\end{equation}   
where $ m_{\langle \lambda \rangle} $ is the multiplicity corresponding to the $ FS_{\langle n \rangle} $-irreducible character $ \chi_{\langle \lambda \rangle} $. Observe also that we can define naturally a left action of $S_{\langle n\rangle}$ on Young multitableaux by
$$
\langle \sigma\rangle T_{\langle \lambda \rangle} = (\sigma_1,\sigma_2,\sigma_3,\sigma_4) (T_{\lambda(1)}, T_{\lambda (2)}, T_{\lambda(3)}, T_{\lambda (4)}) = (\sigma_1 T_{\lambda(1)}, \sigma_2 T_{\lambda (2)}, \sigma_3 T_{\lambda(3)}, \sigma_4 T_{\lambda (4)}),
$$
where $\sigma_i$ acts on $T_{\lambda_{(i)}}$ by the corresponding permutation of elements of $T_{\lambda_{(i)}}$ (see \cite{DR,RV}). We also denote by $Id_i\in S_{n_i}$ the identical permutation of $S_{n_i}$, $i=1,2,3,4$.

Let $ \lambda \vdash n$ and $ T_{\lambda} $ be a Young tableau of shape $ \lambda $. We denote by
$$e_{T_{\lambda}}=\displaystyle\sum_{\overset{\sigma \in R_{T_{\lambda}}}{\tau \in C_{T_{\lambda}}}}(-1)^{\tau}\sigma \tau
$$
the minimal essential idempotent element of $FS_n$ corresponding to $T_{\lambda}$, where $R_{T_{\lambda}}$ and $C_{T_{\lambda}}$ are the subgroups of $ S_n $ stabilizing rows and columns of $ T _ {\lambda} $, respectively.

In what follows, we will naturally identify the group algebra $ FS_{\langle n \rangle} $ with the tensor product of group algebras $ FS_{n_1} \otimes_F FS_{n_2} \otimes_F FS_{n_3} \otimes_F FS_{n_4}$.

The proof of the next two results follows from \cite{RV} and the properties of the tensor product of algebras. We refer the reader to the books \cite{BH,CR,jkerb,sagan} for details.

\begin{lema}\label{etmu}
For every Young multitableau $T_{\langle\mu\rangle}$ of shape $\langle\mu\rangle\vdash \cn$, the element $e_{T_{\langle\mu\rangle}}=e_{T_{\mu (1)}}\otimes e_{T_{\mu (2)}}\otimes e_{T_{\mu (3)}}\otimes e_{T_{\mu (4)}}$ is a minimal essential idempotent of $FS_{\cn}$ and $FS_{\cn}e_{T_{\langle\mu\rangle}}$ is a minimal left ideal of the algebra $FS_{\cn}$ with character $\chi_{\langle\mu\rangle}$. If $T_{\langle\mu\rangle}$ and $T_{\langle\mu\rangle}^*$ are Young multitableaux of the same shape, then $e_{T_{\langle\mu\rangle}}$ and $e_{T^*_{\langle\mu\rangle}}$ are conjugated in $FS_{\cn}$ through some $\langle\sigma\rangle\in FS_{\cn}$. Moreover, $\langle\sigma\rangle e_{T_{\langle \mu\rangle}}\langle\sigma\rangle^{-1} = e_{{\langle\sigma\rangle}T_{\langle\mu\rangle}}=e_{T_{\langle\mu\rangle}^*}$.
\end{lema}

\begin{lema}\label{Met1}
Let $M$ be an irreducible left $FS_{\cn}$-module with character $\chi(M)=\chi_{\langle\mu\rangle}$, $\langle\mu\rangle\vdash \cn$. Then $M$ can be generated as a left $FS_{\cn}$-module by an element of the form $e_{T_{\langle\mu\rangle}}f$ for some $f\in M$ and some Young multitableau  $T_{\langle\mu\rangle}$ of the shape $\langle\mu\rangle$. Moreover, for any Young multitableau $T^*_{\langle\mu\rangle}$ of shape $\langle\mu\rangle$, there exists $f'\in M$ such that $M=FS_{\cn}e_{T^*_{\langle\mu\rangle}}f'$.
\end{lema}

Denote by 
\begin{equation*}
\mathcal{T}_{\langle \mu \rangle} =\{ T_{\langle \mu \rangle} \mbox{ is a Young multitableau of the shape } \ \langle \mu \rangle \}
\end{equation*}
the set of all Young multitableaux of the shape $\langle \mu \rangle$ and 
\begin{eqnarray*}
\mathcal{ST}_{\langle \mu \rangle} =\{ T_{\langle \mu \rangle} \in \mathcal{T}_{\langle \mu \rangle} \  | \  T_{\langle \mu \rangle} \mbox{ is standard} \}
\end{eqnarray*}
the set of all standard Young multitableaux of the shape $\langle \mu \rangle.$ We assume that a Young multitableau $T_{\langle \mu \rangle}=(T_{\mu(1)},T_{\mu(2)},T_{\mu(3)},T_{\mu(4)})$ is standard if $T_{\mu(i)}$ is standard for all $i=1, 2, 3, 4.$

The following proposition is also based on the properties of the tensor product of algebras, classic structure theorems for semisimple finite dimensional associative algebras and classic results of representation theory.

\begin{propo}\label{bilat-id}
Let $F$ be a field of characteristic zero. Then, for any 
$\cn=(n_1,n_2,n_3,n_4)$ ($n_i \geq 0$),  
\begin{equation} \label{dirsumSn}
FS_{\cn}=\bigoplus_{\langle\mu\rangle \vdash \cn} I_{\langle\mu\rangle}
\end{equation}
is the direct sum of minimal two-sided ideals $I_{\langle\mu\rangle}$ of the algebra $FS_{\cn}.$ The minimal two-sided ideal $I_{\langle\mu\rangle}$ corresponding to the multipartition $\langle\mu\rangle \vdash \cn$ is equal to
\begin{eqnarray*}
I_{\langle\mu\rangle}=FS_{\cn} \mathcal{E}_{\langle\mu\rangle}=
\sum_{T_{\langle\mu\rangle} \in \mathcal{ST}_{\langle \mu \rangle}} FS_{\cn} e_{T_{\langle\mu\rangle}}.
\end{eqnarray*}
Moreover, the generator of $I_{\langle\mu\rangle}$, 
$\mathcal{E}_{\langle \mu \rangle}=\mathcal{E}_{\mu(1)} \otimes \mathcal{E}_{\mu(2)} \otimes \mathcal{E}_{\mu(3)} \otimes \mathcal{E}_{\mu(4)}  \in FS_{\cn}$,
is an essential central idempotent of $FS_{\cn},$
where
$$
\mathcal{E}_{\mu(i)} = \displaystyle \sum_{\sigma_i\in S_{n_i}} \chi_{\mu(i)}(\sigma_i) \  \sigma_i \ \in F{S_{n_i}}, \ \ \  i=1,2,3,4.
$$
\end{propo}

\begin{proof}
Formula (\ref{dirsumSn}) is the classic Wedderburn decomposition of a finite dimensional semisimple algebra $FS_{\cn}.$ This is also a well-known result of the classic structure theory of finite dimensional semisimple associative algebras that any minimal two-sided ideal $\tilde{I}$ of $FS_{\cn}$ is the sum of all minimal left ideals of $FS_{\cn},$ which are isomorphic as left $FS_{\cn}$-modules (see, e.g., \cite{jacobs}). Then Lemmas \ref{etmu}, \ref{Met1} and classic results of representation theory of symmetric groups (\cite{RV,jkerb}) immediately imply that 
$$\tilde{I}=I_{\langle\mu\rangle}=
\sum_{T_{\langle\mu\rangle} \in \mathcal{T}_{\langle \mu \rangle} } FS_{\cn} e_{T_{\langle\mu\rangle}}$$ for some multipartition $\langle\mu\rangle \vdash \cn.$ From another side, any multipartition $\langle\mu\rangle \vdash \cn$ corresponds to some minimal two-sided ideal $I_{\langle\mu\rangle}$ of $FS_{\cn}.$

Also it is well known from the $S_n$-representation theory via Young tableaux 
(see e.g. \cite{RV,jkerb}) that the element $\mathcal{E}_{\mu(i)}$ is an essential central idempotent of $FS_{n_i}$, $I_{\mu(i)}=FS_{n_i} \mathcal{E}_{\mu(i)}$ is the minimal two-sided ideal of $FS_{n_i}$ corresponding to the Young diagram of the shape $\mu(i)$ and 
$$I_{\mu(i)}=FS_{n_i} \mathcal{E}_{\mu(i)}=\bigoplus_{T_{\mu(i)} \text{ is  standard}} FS_{n_i} e_{T_{\mu(i)}}$$ (see, for instance, Proposition 2.2.14 of \cite{RV} or \cite{jkerb}), for any $i=1,2,3,4$. 

These facts immediately imply that the element $\mathcal{E}_{\langle \mu \rangle}$ is also an essential central idempotent of $FS_{\langle n \rangle}$ and 
\begin{eqnarray} \label{st}
&&FS_{\cn} \mathcal{E}_{\langle \mu \rangle} = (FS_{n_1} \mathcal{E}_{\mu(1)}) \otimes (FS_{n_2} \mathcal{E}_{\mu(2)}) \otimes (FS_{n_3} \mathcal{E}_{\mu(3)}) \otimes (FS_{n_4} \mathcal{E}_{\mu(4)})= \nonumber \\
&&(\bigoplus_{T_{\mu(1)}  \text{ is }  \atop \text{ standard}} FS_{n_1} e_{T_{\mu(1)}}) \otimes (\bigoplus_{T_{\mu(2)} \text{ is }  \atop \text{ standard}} FS_{n_2} e_{T_{\mu(2)}}) \otimes (\bigoplus_{T_{\mu(3)} \text{ is }  \atop \text{ standard}} FS_{n_3} e_{T_{\mu(3)}}) \otimes (\bigoplus_{T_{\mu(4)} \text{ is }  \atop \text{ standard} } FS_{n_4} e_{T_{\mu(4)}}) \subseteq \nonumber \\
&&\sum_{(T_{\mu(1)},\dots,T_{\mu(4)}) 
\atop \text{ is  standard}}
(FS_{n_1} e_{T_{\mu(1)}}) \otimes (FS_{n_2} e_{T_{\mu(2)}}) \otimes
(FS_{n_3} e_{T_{\mu(3)}}) \otimes (FS_{n_4} e_{T_{\mu(4)}}) =  \nonumber \\
&&\sum_{T_{\langle\mu\rangle} 
 \in \mathcal{ST}_{\langle \mu \rangle}} FS_{\cn} e_{T_{\langle\mu\rangle}} \subseteq \sum_{T_{\langle\mu\rangle} 
 \in \mathcal{T}_{\langle \mu \rangle}} FS_{\cn} e_{T_{\langle\mu\rangle}}=I_{\langle\mu\rangle}.
\end{eqnarray}
Hence, we have $FS_{\cn} \mathcal{E}_{\langle \mu \rangle}$ is a non-zero two-sided ideal of $FS_{\cn}$ ($FS_{\cn} \mathcal{E}_{\langle \mu \rangle} \ni \mathcal{E}_{\langle \mu \rangle} \neq 0$) and $FS_{\cn} \mathcal{E}_{\langle \mu \rangle} \subseteq I_{\langle\mu\rangle}.$ Since $I_{\langle\mu\rangle}$ is a minimal two-sided ideal of $FS_{\cn},$ one has $FS_{\cn} \mathcal{E}_{\langle \mu \rangle} = I_{\langle\mu\rangle}.$ Moreover, from (\ref{st}), one has 
$$
I_{\langle\mu\rangle}=FS_{\cn} \mathcal{E}_{\langle \mu \rangle} \subseteq \sum_{T_{\langle\mu\rangle} 
 \in \mathcal{ST}_{\langle \mu \rangle}} FS_{\cn} e_{T_{\langle\mu\rangle}} \subseteq  I_{\langle\mu\rangle}.
$$
Hence, we have $I_{\langle\mu\rangle}=\sum_{T_{\langle\mu\rangle} 
\in \mathcal{ST}_{\langle \mu \rangle}} FS_{\cn} e_{T_{\langle\mu\rangle}}.$
\end{proof}

Besides this description of simple two-sided ideals of $FS_{\cn}$, we can also prove the following lemma.

\begin{lema}\label{Decomp}
Let $M$ be a left $FS_{\cn}$-module, $N \subseteq M$ its submodule, and $f \in M$ any element of $M.$ Then, $f \in N$ if, and only if, $e_{T_{\langle\mu\rangle}}f \in N$ for all multipartitions $\langle \mu \rangle \vdash \cn$ and all Young multitableaux  
$T_{\langle \mu \rangle}$ of shape $\langle \mu \rangle$.
\end{lema}
\begin{proof}
$N$ is a left $FS_{\cn}$-module. Thus, if $f \in N$, then $e_{T_{\langle\mu\rangle}}f \in N$ for all $\langle \mu \rangle \vdash \cn,$ and all Young multitableaux  
$T_{\langle\mu\rangle}$ of shape $\langle\mu\rangle$ obviously.

Conversely, we have that
$$FS_{\cn}=\sum_{{
\langle\mu\rangle\vdash \cn} \atop
{T_{\langle\mu\rangle} \in \mathcal{ST}_{\langle \mu \rangle}}} FS_{\cn} e_{T_{\langle\mu\rangle}}$$
from Proposition \ref{bilat-id}.
Hence, for an element $f \in M$ we have 
$$f \in FS_{\cn} f=\sum_{{
\langle\mu\rangle\vdash \cn} \atop
{T_{\langle\mu\rangle} \in \mathcal{ST}_{\langle \mu \rangle}} } FS_{\cn} e_{T_{\langle\mu\rangle}} f.$$
Thus, if for all $\langle \mu \rangle \vdash \cn$ and Young multitableaux  
$T_{\langle\mu\rangle}$ of shape $\langle\mu\rangle$, we have $e_{T_{\langle\mu\rangle}}f \in N$, then $f \in FS_{\cn} f \subseteq N$ also holds.
\end{proof}

Now, consider $P_{\cn}$ the space of all multilinear $\#$-superpolynomials in which the first $n_1$ variables are even symmetric, the next $n_2$ variables are even skew, the next $n_3$ variables are odd symmetric and the last $n_4$ variables are odd skew. 

Note that $ P_{\cn} $ has a natural structure of  left $FS_{\cn} $-module, defining the following action of $ S_{\cn} $ on $ P_{\cn} $: given $ f \in P_{\cn} $ and $\langle\sigma\rangle= (\sigma_1, \sigma_2, \sigma_3, \sigma_4) \in S_{\cn} $, one has
\begin{eqnarray*}
\langle\sigma\rangle f=(\sigma_1, \sigma_2, \sigma_3, \sigma_4)f(y_{0,1}, \cdots, y_{0,n_1}, z_{0,1},\cdots, z_{0,n_2}, y_{1,1}, \cdots, y_{1,n_3}, z_{1,1}, \cdots, z_{1,n_4})\\
=f(y_{0,\sigma_1(1)}, \cdots, y_{0,\sigma_1(n_1)}, z_{0,\sigma_2 (1)},\cdots, z_{0,\sigma_2(n_2)}, y_{1,\sigma_3(1)}, \cdots, y_{1,\sigma_3(n_3)}, z_{1,\sigma_4(1)}, \cdots, z_{1,\sigma_4(n_4)}).
\end{eqnarray*}

It is possible to show that $P_{n}^{grs}\cong  \displaystyle\bigoplus_{\cn}
{{n}\choose{\cn}} P_{\cn}$, 
where ${{n}\choose{\cn}}={{n}\choose{n_1,n_2,n_3,n_4}}$ is the multinomial coefficient (see \cite{GID}), i.e. $P_{n}^{grs}$ is the sum of ${n}\choose{\cn}$ its subspaces, which are copies of the space $P_{\cn}$, obtained by renaming variables of elements of $P_{\langle n\rangle}$, more precisely, choosing numbers of any type of $n_i$ variables in the set $\{1,2,...,n\}$ in all possible ways. It is clear that it is enough to study $P_{\langle n\rangle}$ to understand the behavior of multilinear $\#$-superidentities. 

Considering the left $FS_{\cn}$-module 
$$
P_{\cn}(A):=\frac{P_{\cn}}{P_{\cn} \cap Id_2^{\#}(A)}
$$  
and $c_{\cn}(A):= \dim_{F} P_{\cn}(A)$, we have that (see \cite{GID} for a superinvolution and \cite{ARA} for a graded involution)
$$
c_{n}^{grs}(A):= \displaystyle\sum_{n_{1}+\cdots+n_{4}=n} {{n}\choose{\cn}}c_{\cn}(A).
$$

According to (\ref{caracter}), consider $\langle n\rangle $-th character 
$\chi_{\cn} (A) $ (the $\langle n\rangle$-th cocharacter of $\#$-superiden\-ti\-ties of $A$ or simply $\langle n\rangle $-th $\#$-cocharacter of $A$) associated to the left $FS_{\langle n\rangle}$-module $ P_{\cn} (A) $ and its decomposition given by
\begin{equation}\label{caracter2}
\chi_{\cn}(A) = \dis\sum_{\langle\lambda\rangle\vdash \cn}m_{\langle\lambda\rangle}\chi_{\langle\lambda\rangle}.
\end{equation}

\begin{propo}\label{propetm}
Let $A$ be a $\#$-superalgebra with $\langle n\rangle$-th $\#$-cocharacter $\chi_{\cn}(A)$ given in (\ref{caracter2}). If $A$ is a $PI$-algebra, then the multiplicity $m_{\langle \mu\rangle}$ is equal to zero, for a multipartition $\langle \mu\rangle\vdash \cn$, if, and only if, for any Young multitableau $T_{\langle \mu\rangle}$ of shape $\langle\mu\rangle$ and for any multilinear polynomial $f=f(\mathcal{Y}_{0,n_1},\mathcal{Z}_{0,n_2},\mathcal{Y}_{1,n_3},\mathcal{Z}_{1,n_4})\in P_{\cn}$, the algebra $A$ satisfies the $\#$-superidentity $e_{T_{\langle\mu\rangle}}f\equiv 0$.
\end{propo}

\begin{proof}
By Maschke's Theorem, there exists $J$ the submodule of $P_{\cn}$ such that $P_{\cn}=(P_{\cn}\cap Id_2^{\#}(A))\oplus J.$ By the second isomorphism theorem, $J$ is isomorphic to $P_{\cn}(A)$ and, therefore, $\chi_J=\chi_{\cn}(A)$. If $m_{\langle\mu\rangle}=0$, there is no submodule of $J$ isomorphic to $FS_{\cn}e_{T_{\langle\mu\rangle}}$, for any multitableau $T_{\langle\mu\rangle}$ of shape $\langle\mu\rangle$, since, by Lemma \ref{etmu}, the element $FS_{\cn}e_{T_{\langle\mu\rangle}}$ is irreducible. Suppose that $e_{T_{\langle\mu\rangle}}\cdot h\neq 0$ for some $h\in J$, and consider the application $\varphi: FS_{\cn}e_{T_{\langle\mu\rangle}}\to FS_{\cn}e_{T_{\langle\mu\rangle}}\cdot h$ given by $\varphi(\alpha)=\alpha\cdot h$, $\alpha\in FS_{\cn}e_{T_{\langle\mu\rangle}}$. Note that $\varphi$ is an isomorphism of $F S_{\cn}$-modules. This gives us an absurd statement, because $FS_{\cn}e_{T_{\langle\mu\rangle}}\cdot h\subset J$. Hence, $e_{T_{\langle\mu\rangle}}\cdot h=0$ for all $h\in J$ necessarily. Recalling that $P_{\cn}=(P_{\cn}\cap Id_2^{\#}(A))\oplus J$, every element $f\in P_{\cn}$ may be decomposed as $f=g+h$, where $g\in P_{\cn} \cap Id_2^{\#}(A)$ and $h\in J$, such that
$$
e_{T_{\langle\mu\rangle}} f = e_{T_{\langle\mu\rangle}}g + e_{T_{\langle\mu\rangle}} h = e_{T_{\langle\mu\rangle}} g \in P_{\cn}\cap Id_2^{\#}(A). 
$$
Thus, $e_{T_{\langle\mu\rangle}} f\in Id_2^{\#}(A)$ for all $f\in P_{\cn}$. Conversely, suppose that $m_{\langle\mu\rangle}\neq 0$. Then, $J$ possesses a minimal submodule $M$ satisfying $\chi(M)=\chi_{\langle\mu\rangle}$. By Lemma \ref{Met1}, there exists $f\in M$ such that $M=FS_{\cn}e_{T_{\langle\mu\rangle}}f$ for some multitableau $T_{\langle\mu\rangle}$ associated to the multipartition $\langle\mu\rangle$. In view of $M\neq 0$, it follows that $e_{T_{\langle\mu\rangle}}f\neq 0$, and consequently $e_{T_{\langle\mu\rangle}}f\notin (P_{\cn}\cap Id_2^{\#}(A))\cap J=\{0\}$. Since $e_{T_{\langle\mu\rangle}}f\in J$, we conclude that $e_{T_{\langle\mu\rangle}}f\notin Id_2^{\#}(A)$.
\end{proof}

Let $m$ be a positive integer. Consider the square partition $\nu=(m^m)\vdash m^2$ with $T_\nu$ its respective standard tableau. For simplicity, we adopt the following notation for the next remark. We will denote $\mathcal{S}=FS_1\otimes \cdots \otimes F S_{m^2}\otimes \cdots \otimes FS_1$, where $S_1$ is the trivial group, and $\tilde{e}_{T_\nu}=Id_1\otimes\cdots \otimes e_{T_{\nu}}\otimes \cdots \otimes Id_4\in \mathcal{S}$, where $Id_i\in S_{1}$ is the identical permutation.

\begin{obs}\label{obspropS}
Consider $M$ the left $\mathcal{S}$-module generated by the element $\tilde{e}_{T_\nu}$. Then, $M$ is an irreducible $\mathcal{S}$-module (see e.g. Theorem 1.11.3 from \cite{sagan}, which is valid also for any field in the case of symmetric groups). Moreover, for any $\langle t\rangle=(t_1,t_2,t_3,t_4)$ with $t=t_1+t_2+t_3+t_4$ and $t_{i_0}\geq m^2$ for some $i_0\in\{1,2,3,4\}$ and for some $\#$-superalgebra $A$, if $w\in P_{\langle t\rangle}$ is a $\#$-supermonomial such that
$
\tilde{e}_{T_{\nu}}w\not \equiv 0
$
in $A$, then the left $\mathcal{S}$-module $M_0$ generated by the element $\tilde{e}_{T_{\nu}}w$ is irreducible. In fact, note that $M_0=M w$. Let $0\neq N_0\subseteq M_0$ be an $\mathcal{S}$-submodule of $M_0$. We have $N_0=N w$, where $0\neq N\subseteq M$ is an $\mathcal{S}$-submodule of $M$. Since $M$ is irreducible, it follows that $N=M$, showing that $N_0=M_0.$
\end{obs}

In the following, we present two well-known lemmas concerning the action of 
$S_n$ on the space $P_n$ of all ordinary multilinear polynomials of degree $n$. The first one gives us a lower bound for the dimensions of the representations, defined by square partitions.

\begin{lema} ([\citealp{RV}, Lemma 4.5.2]) \label{leme4}
Let $q$ be a positive real number and let $n = m^2> e^4q^2$, where $e$ is the base of the natural logarithms. If $ \lambda = (m^m) \vdash n $ is the square partition of $n$, then $deg \chi_{\lambda}> q^{m^2}$.
\end{lema}

\begin{lema}([\citealp{RV}, Lemma 4.5.3])\label{Lem 2} Let $\lambda\vdash n$ and $\mu \vdash m$, $m\leq n$ and suppose that $\mu \leq \lambda\ (D_{\mu}\subseteq D_{\lambda})$. Consider a Young tableau $T_{\lambda}$ such that the integers $1,2,\cdots, m$ are arranged into  the subtableau $T_{\mu}$. Then, in the group algebra $FS_n$, we can write
$
e_{T_{\lambda}}= a e_{T_{\mu}}b,
$
for some $a,b \in FS_n$.
\end{lema}

Given any integers $ d, l \geq 0 $, we define an infinite hook as being
\begin{center}
$ \mathcal{H}(d, l) = \cup_{n \geq 1}\{\lambda = (\lambda_1, \lambda_2, \cdots) \vdash n | \ \lambda_{d + 1} \leq l \} $.
\end{center}

Remember that, given a $ PI $-algebra $A$ over a field $F$ of characteristic zero, we say that its $ n $-th cocharacter $\chi_n (A)$ lies in the hook $\mathcal{H}(d, l ) $ if in the decomposition \  
$ \chi_n(A) = \sum_{\lambda \vdash n} m_{\lambda} \chi_{\lambda} $
of $ \chi_n (A) $ into the sum of irreducible $ S_n $-characters, all multiplicities $ m_{\lambda} $ are equal to zero for all $ \lambda \notin  \mathcal{H}(d, l), $ and in this case, we have  
$$ 
\chi_n(A) = \displaystyle\sum_{\overset{
\lambda \vdash n}
{\lambda \in \mathcal{H}(d,l)}} m_{\lambda} \chi_{\lambda}. 
$$  
Additionally, we write $\chi_n(A) \subseteq \mathcal{H}(d, l)$ in this context.

{\section{Main Results}}

Let us fix some notation for this section. Consider the following finite sets of variables: 
\begin{center}
$\mathcal{Y}_{0,t_1}= \{y_{0,1}, y_{0,2},\cdots, y_{0,t_1}\},\quad \ \ \mathcal{Z}_{0,t_2}= \{z_{0,1}, z_{0,2},\cdots,z_{0,t_2}\},$ \linebreak $\mathcal{Y}_{1,t_3}= \{y_{1,1}, y_{1,2},\cdots,y_{1,t_3}\},$  \ $\mathcal{Z}_{1,t_4}= \{z_{11}, z_{12},\cdots,z_{1t_4}\}$,
\end{center}
for any integers $t_1,t_2,t_3,t_4 \geq 0$.

Then, we write
\begin{equation*}
f(\mathcal{Y}_{0,t_1},\mathcal{Z}_{0,t_2},\mathcal{Y}_{1,t_3},\mathcal{Z}_{1,t_4})\in P_{\langle t\rangle}
\end{equation*}
meaning that $f$ is a multilinear $\#$-superpolynomial in these sets of variables, where $t=t_1 +t_2+t_3+ t_4$ and $\langle t\rangle = (t_1,t_2,t_3,t_4).$ Furthermore, we write $P_{\langle t\rangle}(R_{0,t_1},L_{0,t_2}, R_{1,t_3},L_{1,t_4})$ to denote the vector space of multilinear $\#$-superpolynomials spanned by all multilinear $\#$-supermonomials in the variables $R_{0,t_1}\subseteq Y_0$, $L_{0,t_2}\subseteq Z_0$, $R_{1,t_3}\subseteq Y_1$ and $L_{1,t_4}\subseteq Z_1$, where $R_{0,t_1}$ is the set of $t_1$ even symmetric variables, $L_{0,t_2}$ consists of $t_2$ even skew variables, $R_{1,t_3}$ consists of $t_3$ odd symmetric variables and $L_{1,t_4}$ consists of $t_4$ odd skew variables. We stress that the elements of these sets are not necessarily the first variables of each homogeneous component.

Given a set $P=\{p_1,p_2,\cdots,p_u\}\subseteq\{1,2,\cdots,n\}$, we denote by $S_P$ the group of all permutations that act only in the set $P$.

The next proposition is the key step in the proof of the hook theorem for superidentities with superinvolution or graded involution.

\begin{propo} \label{prop_w_mon}
Let $A$ be a $\#$-superalgebra, $n_1,n_2,n_3,n_4,m$ positive integers and $i_0 \in \{1,2,3,4\}$ with $n_{i_0}\geq m^2$.  Consider $Q=\{q_1,q_2,\cdots,q_{m^2}\}\subseteq\{1,2,\cdots,n_{i_0}\}$, with $|Q|=m^2$. Let 
$$
\rho=\rho_1\otimes \cdots \otimes \rho_{i_0}\otimes \cdots \otimes \rho_4 \in   FS_{n_1}\otimes \cdots \otimes FS_{Q}\otimes \cdots \otimes FS_{ n_4 }$$
such that
$
\rho \cdot f \not \equiv 0
$
in $A$, where $f=f(\mathcal{Y}_{0,n_1},\mathcal{Z}_{0,n_2},\mathcal{Y}_{1,n_3},\mathcal{Z}_{1,n_4})\in P_{\cn}$ is a multilinear $\#$-superpolynomial, $\cn = (n_1,n_2,n_3,n_4)$, $\rho_{i_0}\in FS_Q$ and $\rho_j\in FS_{n_j}$ for all $j\neq i_0$. Then, there exist integers $t_1, t_2, t_3, t_4\geq 0$ and a $\#$-supermonomial $w\in P_{\langle t\rangle}(\mathcal{Y}_{0,t_1},\cdots, X_{t_{i_0}},\cdots, \mathcal{Z}_{1,t_4})$, with $t=t_1+t_2+t_3+t_4$ and $X_{t_{i_0}}=\{y_{\gamma,q_1},\cdots,y_{\gamma,q_{m^2}},\hat{y}_{\gamma,m^2+1},\cdots, \hat{y}_{\gamma,t_{i_0}}\}\subset Y_{\gamma}$ such that $\{\hat{y}_{\gamma,m^2+1},\cdots, \hat{y}_{\gamma,t_{i_0}}\}\subseteq \mathcal{Y}_{\gamma,t_{i_0}}\setminus \{y_{\gamma,q_1},\cdots, y_{\gamma,q_{m^2}}\}$ with $\gamma=0$ if $i_0=1$ and $\gamma=1$ if $i_0=3$, or $X_{t_{i_0}}=\{z_{\gamma,q_1},\cdots,z_{\gamma,q_{m^2}},\hat{z}_{\gamma,m^2+1},\cdots, \hat{z}_{\gamma,t_{i_0}}\}\subset Z_{\gamma}$ such that $\{\hat{z}_{\gamma,m^2+1},\cdots, \hat{z}_{\gamma,t_{i_0}}\}\subseteq \mathcal{Z}_{\gamma,t_{i_0}}\setminus \{z_{\gamma,q_1},\cdots, z_{\gamma,q_{m^2}}\}$ with $\gamma=0$ if $i_0=2$ and $\gamma=1$ if $i_0=4$ satisfying
$$
(Id_1\otimes \cdots \otimes \rho_{i_0}\otimes \cdots \otimes Id_4) w\not \equiv 0
$$ 
in $A$ with $|X_{t_{i_0}}|=t_{i_0}\geq m^2$ and $ m^2\leq t \leq 2m^2+1$.
\end{propo} 
\begin{proof}
Since
$$
(Id_1\otimes \cdots \otimes \rho_{i_0}\otimes \cdots \otimes Id_4) \cdot (\rho_1\otimes \cdots \otimes Id_{i_0}\otimes \cdots \otimes \rho_4) f = \rho \cdot 
f\not \equiv 0
$$
in $A$, necessarily, 
$$
(Id_1\otimes \cdots \otimes \rho_{i_0}\otimes \cdots \otimes Id_4)\tilde{f} \not \equiv 0
$$
in $A$, where $\tilde{f}= (\rho_1\otimes \cdots \otimes Id_{i_0}\otimes \cdots \otimes \rho_4) f$ is a multilinear $\#$-superpolynomial in $P_{\cn}$. Thus, there exists a $\#$-supermonomial $g$ of $\tilde{f}$ such that
$$
(Id_1\otimes \cdots \otimes \rho_{i_0}\otimes \cdots \otimes Id_4)g \not \equiv 0
$$
in $A$. Let us highlight the variables of $g$ which are under the action of $\rho_{i_0}$. Without loss of generality, we may assume $i_0=1$. By hypothesis, $n_1\geq m^2$. We write $g$ as: 
$$
g=w_{0}y_{0,q_1}w_1y_{0,q_2}\cdots w_{k-1}y_{0,q_k}w_k,
$$
where $k=m^2$, $\{q_1,\cdots, q_k\}\subseteq\{1,\cdots,n_1\}$ and $w_0,w_1,\cdots, w_k$ are multilinear $\#$-supermono\-mi\-als, possibly empty, in the variables of the set $\mathcal{Z}_{0,n_2} \cup \mathcal{Y}_{1,n_3} \cup \mathcal{Z}_{1,n_4} \cup (\mathcal{Y}_{0,n_1}\setminus\{y_{0,q_{1}},\cdots,y_{0,q_k}\})$. 
It is clear that any $\#$-supermonomial $w_p$ is $\mathbb{Z}_2$-homogeneous of a degree $\gamma_p$ for all $p=0,\dots,k.$ Thus, 
when we evaluate the variables of a $\#$-supermonomial $w_p$ in the elements of $A_0^+\cup A_0^-\cup A_1^+\cup A_1^-$ and analyze its homogeneous degree, we obtain an evaluation $\hat{w}_p$ of $w_p$ such that $\hat{w}_p\in A_0$ if $\gamma_p=0$ or $\hat{w}_p\in A_1$ if $\gamma_p=1,$ i.e. $\deg w_p=\deg \hat{w}_p = \gamma_p.$ Hence, we may decompose $\hat{w}_p$ in $\hat{w}_p=\hat{u}_p^{\gamma_p} + \hat{v}_p^{\gamma_p}$, where $\hat{u}_p^{\gamma_p}\in A_{\gamma_p}^+$ and $\hat{v}_p^{\gamma_p}\in A_{\gamma_p}^-$ are the symmetric and skew parts of $\hat{w}_p$, respectively. Let's replace in $g$ the nonempty $\#$-supermonomials $w_p$ by $y_p^{\gamma_p}+z_p^{\gamma_p}$, where $y_p^{0}\in \mathcal{Y}_{0,n_1}\setminus\{y_{0,q_{1}},\cdots,y_{0,q_k}\},$ $y_p^{1}\in \mathcal{Y}_{1,k+1},$ $z_p^{\gamma_p}\in\mathcal{Z}_{\gamma_p,k+1}$, with $\gamma_p=0$ or 1, and
all the variables $\{ y_p^{\gamma_p}, z_p^{\gamma_p}, p=0,\dots,k \}$ are pairwise different. Therefore,
we obtain a new $\#$-superpolynomial $\tilde{g}$ of the form
\begin{equation}\label{eqmon2}
\tilde{g}= (y_0^{\gamma_0}+z_0^{\gamma_0})y_{0,q_1}\cdots (y_{k-1}^{\gamma_{k-1}}+z_{k-1}^{\gamma_{k-1}})y_{0,q_k}(y_k^{\gamma_k}+z_k^{\gamma_k}).
\end{equation}
We denote $u_p^{\gamma_p}=\dfrac{w_p+w_p^\#}{2}$ and $v_p^{\gamma_p}=\dfrac{w_p-w_p^\#}{2}$ the $\mathbb{Z}_2$-homogeneous $\#$-superpolynomials of degree $\gamma_p$ that depend of the same variables of $w_p$. We have that $(u_p^{\gamma_p})^\#=u_p^{\gamma_p}$, $(v_p^{\gamma_p})^\#=-v_p^{\gamma_p}$ and $w_p=u_p^{\gamma_p}+v_p^{\gamma_p}$. It is clear that the $\#$-supermonomial $g$ is obtained from $\tilde{g}$ by replacement $y_p^{\gamma_p}=u_p^{\gamma_p}$ and $z_p^{\gamma_p}=v_p^{\gamma_p}$ in $\mathcal{F}=F\langle X | \mathbb{Z}_2, \# \rangle$. We highlight that this replacement is a graded endomorphism of $\mathcal{F}$ and $\#$-invariant. Since $g$ is obtained from $\tilde{g}$ by suitable replacements of the variables which are different of those belonging to the set $\{y_{0,q_1}\cdots,y_{0,q_{m^2}}\}$, in which $\rho_{i_0}=\rho_1$ acts, then the $\#$-superpolynomial $(\rho_1\otimes Id_2 \otimes Id_3 \otimes Id_4) g$ also is obtained from $(\rho_1\otimes Id_2 \otimes Id_3 \otimes Id_4) \tilde{g}$ by the same replacements. In view of 
$$
(\rho_1\otimes Id_2 \otimes Id_3 \otimes Id_4) g \not \equiv 0
$$
in $A$ and $g$ is obtained from $\tilde{g}$ by suitable replacements of variables, we have
$$
(\rho_1\otimes Id_2 \otimes Id_3 \otimes Id_4)\tilde{g} \not \equiv 0
$$
in $A$. Since $\tilde{g}$ is a multilinear $\#$-superpolynomial, which is sum of
$\#$-supermonomials of the complete degree at least $m^2$ and at most $2m^2+1$, necessarily, there exist integers $t_1,t_2,t_3,t_4\geq 0$ and a $\#$-supermonomial $w=w(X_{t_1}, \mathcal{Z}_{0,t_2},\mathcal{Y}_{1,t_3},\mathcal{Z}_{1,t_4})\in P_{\langle t \rangle}(X_{t_1}, \mathcal{Z}_{0,t_2},\mathcal{Y}_{1,t_3},\mathcal{Z}_{1,t_4})$, with $t=t_1+t_2+t_3+t_4$ and $X_{t_{1}}=\{y_{0,q_1},\cdots,y_{0,q_{m^2}},\hat{y}_{0,m^2+1},\cdots, \hat{y}_{0,t_{1}}\}\subset Y_{0}$ such that $\{\hat{y}_{0,m^2+1},\cdots, \hat{y}_{0,t_{1}}\}\subseteq \mathcal{Y}_{0,t_{1}}\setminus \{y_{0,q_1},\cdots, y_{0,q_{m^2}}\}$, satisfying $m^2\leq t\leq 2m^2+1$ and
$$
(\rho_1\otimes Id_2 \otimes Id_3 \otimes Id_4) w \not \equiv 0
$$
in $A.$ Observe that $w$ can be chosen as one of the monomials of $\tilde{g}$ possibly with renamed variables. This proves the proposition for the case $i_0=1$.  The other cases $i_0=2,3$ and $4$ are proved in the same way.
\end{proof}

\begin{obs} In the previous proof, the permutation $\rho \in   FS_{n_1}\otimes \cdots \otimes FS_{m^2 }\otimes \cdots \otimes FS_{ n_4 }$ acts on the $\#$-superpolynomial $f\in P_{\cn}$ exchanging the $n_1, \cdots, m^2,\cdots, n_4$ variables and fixing the remaining ones respectively.
\end{obs} 

\begin{exem}\rm
Let $m^2=2^2$, $n=13$, $n_1=3$, $n_2=7$, $n_3=1$, $n_4=2$ and
\begin{eqnarray*}
f & = & y_{0,1}z_{0,1}y_{0,2}z_{0,2}z_{0,3}y_{0,3}z_{0,4}z_{0,5}z_{0,6}z_{0,7}y_{1,1}z_{1,1}z_{1,2} \\
& & - z_{0,2}y_{0,2}y_{0,1}z_{0,4}y_{0,3}z_{0,1}z_{0,6}z_{0,3}z_{1,1}z_{0,7}z_{1,2}y_{1,1}z_{0,5} \in P_{\langle n \rangle}.
\end{eqnarray*}
Assume that $\#$ is a superinvolution.
We aim to determine a $\#$-supermonomial given in Proposition \ref{prop_w_mon} for some $\#$-superalgebra $A,$ and $\rho \in FS_3 \otimes FS_7 \otimes FS_1 \otimes FS_2$. Let's assume that $i_0=2$ then $n_2=7\geq 2^2 = m^2$, with $m=2,$ and choose $Q=\{2,3,4,6\}$. For example, let $\rho \in S_{3}\otimes S_{Q}\otimes S_{1} \otimes S_{2}$ be the element
$$
\rho = (123) \otimes (236) \otimes (1) \otimes (12).
$$
Then, $\rho_{i_0}=\rho_2=(236)$ acts on the indeterminates $\{z_{0,2}, z_{0,3}, z_{0,4}, z_{0,6}\}$.
Suppose that $\rho \cdot f \not \equiv 0$ in $A$. Then, we highlight the indeterminates $\{z_{0,2}, z_{0,3}, z_{0,4}, z_{0,6}\}$ in bold for emphasis and obtain
\begin{eqnarray*}
\rho \cdot f &=& [(1) \otimes (236) \otimes (1) \otimes (1)] \cdot \underbrace{[(123) \otimes (1) \otimes (1) \otimes (12)] f}_{\tilde{f}} \\
& = & [(1) \otimes (236) \otimes (1) \otimes (1)] \tilde{f} \\
& = & [(1) \otimes (236) \otimes (1) \otimes (1)] \left( y_{0,2}z_{0,1}y_{0,3}\mathbf{z}_{0,2}\mathbf{z}_{0,3}y_{0,1}\mathbf{z}_{0,4}z_{0,5}\mathbf{z}_{0,6}z_{0,7}y_{1,1}z_{1,2}z_{1,1} \right. \\
&& \left. - \mathbf{z}_{0,2}y_{0,3}y_{0,2}\mathbf{z}_{0,4}y_{0,1}z_{0,1}\mathbf{z}_{0,6}\mathbf{z}_{0,3}z_{1,2}z_{0,7}z_{1,1}y_{1,1}z_{0,5} \right) \not \equiv 0
\end{eqnarray*}
in $A$. Then, for one of the $\#$-supermonomials
$$u_1 = y_{0,2}z_{0,1}y_{0,3}\mathbf{z}_{0,2}\mathbf{z}_{0,3}y_{0,1}\mathbf{z}_{0,4}z_{0,5}\mathbf{z}_{0,6}z_{0,7}y_{1,1}z_{1,2}z_{1,1}$$
or
$$u_2 = \mathbf{z}_{0,2}y_{0,3}y_{0,2}\mathbf{z}_{0,4}y_{0,1}z_{0,1}\mathbf{z}_{0,6}\mathbf{z}_{0,3}z_{1,2}z_{0,7}z_{1,1}y_{1,1}z_{0,5}$$
of the $\#$-superpolynomial $\tilde{f}$, we have that $[(1) \otimes (236) \otimes (1) \otimes (1)] u_i \not \equiv 0$ in $A$ for $i=0$ or $i=1$. Suppose, for example, that $[(1) \otimes (236) \otimes (1) \otimes (1)] u_1 \not \equiv 0$ in $A$. Therefore, considering the same notation as in Proposition \ref{prop_w_mon}, we have $g = u_1$. Now, let's represent $g$ as follows:
$$
g = w_0 \mathbf{z}_{0,2} w_1 \mathbf{z}_{0,3} w_2 \mathbf{z}_{0,4} w_3 \mathbf{z}_{0,6} w_4,
$$
where
$$
w_0 = y_{0,2} z_{0,1} y_{0,3} \in \mathcal{F}_0, \quad w_1 = \underline{\phantom{a1}} \ (\text{the empty word}), \quad w_2 = y_{0,1} \in \mathcal{F}_0,
$$
$$
 w_3 = z_{0,5} \in \mathcal{F}_0, \quad w_4 = z_{0,7} y_{1,1} z_{1,2} z_{1,1} \in \mathcal{F}_1.
$$
The $\#$-supermonomials $w_p$, with $p = 0, 1, 2, 3, 4$, have respective degrees in the $\mathbb{Z}_2$-grading of the free $\#$-superalgebra $\mathcal{F} = \mathcal{F}_0 \oplus \mathcal{F}_1$.

By decomposing each non-empty word $w_p$, for $0 \leq p \leq 2^2$, into its symmetric and antisymmetric parts and then substituting them with new indeterminates, we obtain a multilinear $\#$-superpolynomial $\tilde{g}$ composed of $\#$-supermonomials of degree $8$ (the degree is at least $4=2^2$ and at most $9 = 2(2^2) + 1$). It follows that
$$
\tilde{g} = (y_{0}^{0} + z_{0}^{0}) \mathbf{z}_{0,2} \mathbf{z}_{0,3} (y_{2}^{0} + z_{2}^{0}) \mathbf{z}_{0,4} (y_{3}^{0} + z_{3}^{0}) \mathbf{z}_{0,6} (y_{4}^{1} + z_{4}^{1}),
$$
and from the fact that
\begin{eqnarray*}
h = [(1) \otimes (236) \otimes (1) \otimes (1)] g \not \equiv 0
\end{eqnarray*}
in $A$, we have that
\begin{eqnarray*}
\tilde{h} = [(1) \otimes (236) \otimes (1) \otimes (1)] \tilde{g} \not \equiv 0
\end{eqnarray*}
in $A$ as well, where $\rho_2=(2,3,6)$ acts only in the variables $\{ z_{0,2}, z_{0,3}, z_{0,4}, z_{0,6}\}.$ Otherwise (if $\tilde{h} \equiv 0$ in $A$), we would get $h \equiv 0$ in $A$ since $h$ is obtained from $\tilde{h}$ by the following substitution:
\begin{eqnarray*}
&& y_{0}^{0} = \displaystyle\frac{y_{0,2} z_{0,1} y_{0,3} - y_{0,3} z_{0,1} y_{0,2}}{2}, \quad z_{0}^{0} = \displaystyle\frac{y_{0,2} z_{0,1} y_{0,3} +y_{0,3} z_{0,1} y_{0,2}}{2}, \\
&& y_{2}^{0} = y_{0,1}, \quad z_{2}^{0} = 0, \quad y_{3}^{0} = 0, \quad z_{3}^{0} = z_{0,5},\\
&& y_{4}^{1} = \displaystyle\frac{z_{0,7} y_{1,1} z_{1,2} z_{1,1} + z_{1,1} z_{1,2} y_{1,1} z_{0,7}}{2}, \quad z_{4}^{1} = \displaystyle\frac{z_{0,7} y_{1,1} z_{1,2} z_{1,1} - z_{1,1} z_{1,2} y_{1,1} z_{0,7}}{2}, 
\end{eqnarray*}
since
\begin{eqnarray*}
&& w_0^{\#} = (y_{0,2} z_{0,1} y_{0,3})^{\#} = (y_{0,3})^{\#} (z_{0,1})^{\#} (y_{0,2})^{\#} = y_{0,3} (-z_{0,1}) y_{0,2} = - y_{0,3} z_{0,1} y_{0,2}, \\
&& w_{2}^{\#} = (y_{0,1})^{\#} = y_{0,1}, \quad w_3^{\#} = (z_{0,5})^{\#} = - z_{0,5},
\end{eqnarray*}
and
\begin{eqnarray*}
w_4^{\#} & = & (z_{0,7} y_{1,1} z_{1,2} z_{1,1})^{\#} = - (z_{1,1})^{\#} (z_{1,2})^{\#} (y_{1,1})^{\#} (z_{0,7})^{\#} \\
& = & - (- z_{1,1}) (- z_{1,2}) y_{1,1} (- z_{0,7}) = z_{1,1} z_{1,2} y_{1,1} z_{0,7}.
\end{eqnarray*}
From the fact that $\tilde{h} = [(1) \otimes (236) \otimes (1) \otimes (1)] \tilde{g} \not \equiv 0$ in $A$, it follows that for one of the monomials $w$ in $\tilde{g}$, we also have $[(1) \otimes (236) \otimes (1) \otimes (1)] w \not \equiv 0$ in $A$. Let's assume, for example, that
\begin{eqnarray*}
w & = & y_{0}^{0} \mathbf{z}_{0,2} \mathbf{z}_{0,3} y_{2}^{0} \mathbf{z}_{0,4} z_{3}^{0} \mathbf{z}_{0,6} y_{4}^{1} \\
& = & y_{0,1} \mathbf{z}_{0,2} \mathbf{z}_{0,3} y_{0,2} \mathbf{z}_{0,4} z_{0,1} \mathbf{z}_{0,6} y_{1,1} \in P_{(2,5,1,0)}(\mathcal{Y}_{0,2}, X_2, \mathcal{Y}_{1,1}, \mathcal{Z}_{1,0}),
\end{eqnarray*}
considering that $y_{0}^{0} = y_{0,1}$, $y_{2}^{0} = y_{0,2}$, $z_{3}^{0} = z_{0,1} \notin \{\mathbf{z}_{0,2}, \mathbf{z}_{0,3}, \mathbf{z}_{0,4}, \mathbf{z}_{0,6}\}$, and $y_{4}^{1} = y_{1,1}$. Here, using the notation from Proposition \ref{prop_w_mon}, we have $\mathcal{Y}_{0,2} = \{ y_{0,1}, y_{0,2} \}$, $\mathcal{Y}_{1,1} = \{ y_{1,1} \}$, $X_2 = \{ \mathbf{z}_{0,2}, \mathbf{z}_{0,3}, \mathbf{z}_{0,4}, \mathbf{z}_{0,6}, z_{0,1} \}$ (with $\rho_2 = (236)$ acting on the indeterminates $\{ \mathbf{z}_{0,2}, \mathbf{z}_{0,3}, \mathbf{z}_{0,4}, \mathbf{z}_{0,6} \}$), $\mathcal{Z}_{1,0} = \emptyset$ (empty set), $t_1 = |\mathcal{Y}_{0,2}| = 2$, $t_2 = |X_2| = 5 \geq 4=m^2$, $t_3 = |\mathcal{Y}_{1,1}| = 1$, $t_4 = |\mathcal{Z}_{1,0}| = 0$, and $t = t_1 + t_2 + t_3 + t_4 = 8$ (since $2^2 = 4 \leq 8 \leq 2(2)^2 + 1$, with $m = 2$, $t = 8$). For the $\#$-supermonomial $w$, we have that
\begin{eqnarray*}
[(1) \otimes (236) \otimes (1) \otimes (1)] w \neq 0
\end{eqnarray*}
in $A$, as stated in Proposition \ref{prop_w_mon}.
\end{exem}

Finally, we are able to prove the main result of this paper. We will write $\lambda \in \mathcal{H}(d,l)$ for some integers $d,l\geq 0$ if the corresponding Young diagram $D_\lambda$ is contained in $\mathcal{H}(d,l)$ and denote $D_\lambda\subseteq \mathcal{H}(d,l)$.

We use also notation
$\langle d,l \rangle=(d_1,l_1;d_2,l_2;d_3,l_3;d_4,l_4)$ for nonnegative integers
$d_i, l_i \geq 0,$ \ $i=1,\dots,4.$

\begin{teo}\label{teoprin}
{\textit (The hook theorem for $\#$-superalgebras)} Let $A$ be a $\#$-superalgebra. If $A$ is a $PI$-algebra, then there exist integers $d_i, l_i\geq 0$, with $i= 1,2,3, 4$, such that the $\cn$-th cocharacter of $\#$-superidentities, $\chi_{\cn}(A)$, is contained in a quadruple hook $\mathcal{H}_{\langle d, l \rangle}\langle n \rangle= (\mathcal{H}(d_1, l_1),\mathcal{H}(d_2, l_2), \mathcal{H}(d_3, l_3), \mathcal{H}(d_4, l_4) )$, that is, 
$$
\chi_{\langle n\rangle}(A)= \displaystyle\sum_{\overset{
\langle \lambda \rangle\vdash \cn}
{\langle \lambda \rangle \in \mathcal{H}_{\langle d, l \rangle}\langle n \rangle}}
m_{\langle \lambda \rangle}\chi_{\langle \lambda \rangle}.
$$
\end{teo}

\begin{proof}
By hypothesis, $A$ is a $PI$-algebra. From Corollary \ref{corregv}, there exists $\widehat{d} \geq 1$ such that $c_n^{grs}(A)\leq (\widehat{d})^n$ for all $n\geq 1$. Consider $q=(\widehat{d})^{3}$ and an integer $m$ such that $e^2q + 1\leq m\leq e^2q + 2$, where $e$ denotes the base of the natural logarithm. Let us prove that, for any $n\geq 1$, the $\langle n\rangle$-th $\#$-cocharacter of $A$ is contained in the quadruple square hook $(\mathcal{H}(m, m),\mathcal{H}(m, m), \mathcal{H}(m, m), \mathcal{H}(m, m))$. 

Suppose, by contradiction, that there exists $n\geq 1,$ and
$\cn=(n_1,\dots,n_4)$ with $n_1+\dots+n_4=n,$ such that 
\begin{equation}\label{=cocara}
\chi_{\langle n \rangle}(A)= \displaystyle\sum_{\langle \lambda \rangle\vdash \cn}m_{\langle \lambda \rangle}\chi_{\langle \lambda \rangle}, 
\end{equation}
and $m_{\langle \mu \rangle}\neq 0$ for some multipartition $\langle \mu \rangle \notin (\mathcal{H}(m, m),\mathcal{H}(m, m), \mathcal{H}(m, m), \mathcal{H}(m, m))$. Hence, at least one of the $\mu_{(i)}'s$ of  $\langle \mu \rangle=(\mu(1),\dots,\mu(4))$ does not belong to $\mathcal{H}(m, m)$. Say $D_{\mu(i_0)} \nsubseteq \mathcal{H}(m, m)$, where $i_0\in \{1,2,3,4\}$. Thus, $D_{\mu(i_0)}$ contains the square diagram $D_{\nu_{i_0}}$, where $\nu_{i_0}= (m^{m})\vdash m^{2}$. We have that $\mu(i_0)\geq \nu_{i_0}$ ($D_{\nu_{i_0}} \subseteq D_{\mu(i_0)}$). Since $m_{\langle \mu \rangle}\neq 0$ and $A$ is a $PI$-algebra, it follows from Proposition \ref{propetm} that there exist a multitableau $T_{\langle \mu \rangle}$, with $\langle \mu\rangle\vdash \cn$, and a $\#$-superpolynomial $f\in P_{\langle n \rangle}$ such that $e_{T_{\langle \mu \rangle}}f \notin Id_2^{\#}(A)$. Once  
$$
e_{T_{\langle \mu \rangle}}f=(e_{T_{\mu(1)}}\otimes e_{T_{\mu(2)}}\otimes e_{T_{\mu(3)}}\otimes e_{T_{\mu(4)}})f\notin Id_2^{\#}(A),
$$
there exists a multilinear $\#$-supermonomial $m_1=m_1 (\mathcal{Y}_{0,n_1},\mathcal{Z}_{0,n_2},\mathcal{Y}_{1,n_3},\mathcal{Z}_{1,n_4})\in P_{\langle n \rangle} $ such that \linebreak $e_{T_{\langle \mu \rangle}}m_1 \notin Id_2^{\#}(A)$ ($m_1$ is one of the monomials of $f$).

Now, consider a quadruple of permutations $\langle \sigma \rangle_{i_0}= (Id_1, \cdots, \sigma_{i_0}, \cdots, Id_4)\in S_{n_1}\times \cdots \times S_{n_{i_0}}\times \cdots \times S_{n_4}$ ($\sigma_{i_0}\in S_{n_{i_0}}$) such that in the multitableau $\langle \sigma \rangle_{i_0} T_{\langle\mu \rangle}$ the tableau $\sigma_{i_0} T_{\mu(i_0)}$ has the entries from $1$ to $m^2$ in the positions of the square diagram $D_{\nu_{i_0}} \subseteq D_{\mu(i_0)}$. Then we have the corresponding element $e_{\langle \sigma \rangle_{i_0} T_{\langle \mu \rangle}}={\langle \sigma \rangle}_{i_0} e_{T_{\langle \mu \rangle}}{\langle \sigma \rangle}_{i_0}^{-1}$ of $FS_{\cn}$, where 
$\langle \sigma \rangle_{i_0}^{-1}= (Id_1, \cdots, \sigma_{i_0}^{-1}, \cdots, Id_4)$. It implies that
$$
\langle\sigma \rangle_{i_0}^{-1}e_{\langle \sigma \rangle_{i_0} T_{\langle \mu \rangle}}{\langle \sigma \rangle}_{i_0} m_1= e_{T_{\langle \mu \rangle}} m_1 \notin Id_2^{\#}(A).
$$
Denoting $m'=\langle \sigma \rangle_{i_0} m_1$, we have that $\langle \sigma \rangle_{i_0}^{-1} e_{\langle \sigma \rangle_{i_0} T_{\langle \mu \rangle}} m'\notin Id_2^{\#}(A)$ and, hence, $e_{\langle \sigma \rangle_{i_0} T_{\langle \mu \rangle}} m' \notin Id_2^{\#}(A)$. We denote by $T_{\nu_{i_0}}$ the standard tableau of the diagram $D_{\nu_{i_0}}$. Using the fact that, in the multitableau $\widetilde{T}_{\langle\mu\rangle} = \langle\sigma\rangle_{i_0} T_{\langle\mu\rangle}=(T_{\mu(1)},\cdots,\sigma_{i_0}T_{\mu(i_0)},\cdots,T_{\mu(4)})$, the tableau $\sigma_{i_0}T_{\mu(i_0)}$ contains the numbers $1,2,\cdots,m^2,$ which are the entries of the tableau $T_{\nu_{i_0}}$, by Lemma \ref{Lem 2}, there exist $a, b \in FS_{n_{i_0}}$ such that 
$$
e_{\sigma_{i_0} T_{ \mu(i_0)}}= a e_{T_{\nu_{i_0}}}b.
$$
Thus, the $\#$-superpolynomial
\begin{eqnarray*}
e_{\widetilde{T}_{\langle\mu \rangle}}m' & = & (e_{T_{\mu(1)}}\otimes\cdots\otimes e_{\sigma_{i_0}T_{\mu(i_0)}}\otimes\cdots e_{T_{\mu(4)}})m'
\end{eqnarray*}
$$
=(Id_1\otimes \cdots \otimes a \otimes \cdots \otimes Id_4) (e_{T_{\mu(1)}}\otimes\cdots \otimes e_{T_{\nu_{i_0}}}\otimes \cdots \otimes e_{T_{\mu(4)}} )(Id_1 \otimes \cdots \otimes b\otimes \cdots\otimes Id_4) m' 
$$
does not belong to $Id_2^{\#}(A)$ and, then, $(e_{T_{\mu(1)}}\otimes\cdots \otimes e_{T_{\nu_{i_0}}}\otimes \cdots \otimes e_{T_{\mu(4)}}) f_1 \notin Id_2^{\#}(A)$, where $f_1=(Id_1 \otimes \cdots \otimes b \otimes \cdots \otimes Id_4) m'\in P_{\cn}$, because $Id_2^\#(A)$ is a left $FS_{\cn}$-module. It follows that

\begin{equation}\label{eqmon1}
(Id_1\otimes\cdots \otimes e_{T_{\nu_{i_0}}}\otimes \cdots \otimes Id_4)f_2\notin Id_2^{\#}(A),
\end{equation}
where $f_2=(e_{T_{\mu(1)}}\otimes\cdots \otimes Id_{i_0}\otimes \cdots \otimes e_{T_{\mu(4)}})f_1\in P_{\cn}$. Observe that $n_{i_0}\geq m^2$ and the element $e_{T_{\nu_{i_0}}}$ acts in the variables with indexes in the set $\{1,\cdots,m^2\}$ of the correspondent type of the $\#$-superpolynomial $f_2$ (the entries of tableau $T_{\nu_{i_0}}$). By Proposition \ref{prop_w_mon}, 
we may find integers $t_1,t_2,t_3,t_4 \geq 0$, with $t_{i_0}\geq m^2$, $t=t_1+t_2+t_3+t_4$, $m^2\leq t\leq 2m^2+1$, and a multilinear $\#$-supermonomial $w= w(\mathcal{Y}_{0,t_1},\mathcal{Z}_{0,t_2},\mathcal{Y}_{1,t_3},\mathcal{Z}_{1,t_4})\in P_{\langle t \rangle}$ such that
\begin{equation}\label{eqnot1}
\widetilde{w}=(Id_1\otimes\cdots \otimes e_{T_{\nu_{i_0}}}\otimes \cdots \otimes Id_4) w \notin Id_2^{\#}(A).
\end{equation}

Let us consider $P_{\langle t \rangle}$ as a left $(FS_1 \otimes \cdots \otimes FS_{m^2 }\otimes \cdots \otimes FS_1)$-module, where $S_1$ is the trivial group, requiring that $FS_1 \otimes \cdots \otimes FS_{m^2 }\otimes \cdots \otimes FS_1$ acts respectively in the first $t_1,\cdots,m^2,\cdots, t_4$ variables of the corresponding type and fixes the others if they exist ($FS_1$ acts trivially in $t_j$ variables, $j \neq i_0$). Now, let $M$ be the $(FS_1 \otimes \cdots \otimes FS_{m^2 }\otimes \cdots \otimes FS_1)$-submodule of $P_{\langle t\rangle}$ generated by the element $\widetilde{w}.$
Therefore, $M \oplus (P_{\langle t\rangle}\cap Id_2^{\#}(A))$ is also an $(FS_1 \otimes \cdots \otimes FS_{m^2 }\otimes \cdots \otimes FS_1)$-submodule of $P_{\langle t\rangle}$. It follows that
\begin{eqnarray*}
\dim (P_{\langle t \rangle}) \geq \dim(M \oplus P_{\langle t\rangle}\cap Id_2^{\#}(A)) = \dim(M) + \dim(P_{\langle t\rangle}\cap Id_2^{\#}(A)).
\end{eqnarray*}
By Remark \ref{obspropS}, the module $M$ is irreducible and, by (\ref{eqnot1}), we have also that $M \not\subseteq (P_{\langle t\rangle}\cap Id_2^{\#}(A))$ (in particular, that is why the sum $M \oplus (P_{\langle t\rangle}\cap Id_2^{\#}(A))$ is direct). It provides also
\begin{eqnarray}\label{eq34}
c_{\langle t\rangle}(A) & = & \dim(P_{\langle t\rangle} / (P_{\langle t\rangle}\cap Id_2^{\#}(A))) \geq \dim(M).
\end{eqnarray}

Since $\dim M = \chi_{\nu_{i_0}}(1)$, where $\nu_{i_0}=(m^m)\vdash m^2$, and $m^2\geq (e^2q + 1)^2 > e^4q^2$, it follows from Lemma \ref{leme4} and  (\ref{eq34}) that
\begin{equation}\label{Eqct1}
c_{\langle t\rangle}(A) \geq \dim{M}=\chi_{\nu_{i_0}}(1)> q^{m^2}.
\end{equation}
Hence, from 
$$
c_t^{grs}(A) = \displaystyle\sum_{\langle \tilde{t}\rangle} \left(\begin{array}{c}
t\\
\langle \tilde{t}\rangle
\end{array}\right)c_{\langle \tilde{t}\rangle}(A) \geq c_{\langle t\rangle}(A)>q^{m^2}=(\widehat{d})^{3m^2},
$$
and $m^2\geq \dfrac{t-1}{2}$ (due to Proposition \ref{prop_w_mon} we have $t \leq 2m^2+1$), one has
$c_t^{grs}(A) > (\widehat{d})^{\frac{3}{2}(t-1)}.$
In view of \ $\dfrac{3}{2}(t-1) > t$, for $t\geq m^2\geq e^4 > 3$, it follows that
$$
c_t^{grs}(A) > (\widehat{d})^t.
$$
We obtain a contradiction with Corollary \ref{corregv}, since $c_n^{grs}(A)\leq (\widehat{d})^n$ for all $n\geq 1$. Therefore, all multiplicity $m_{\langle \mu\rangle}$ in (\ref{=cocara}) must be equal to zero if $\langle \mu\rangle \notin (\mathcal{H}(m, m),\mathcal{H}(m, m), \mathcal{H}(m, m), \mathcal{H}(m, m) )$, and the proof of the theorem is concluded.
 
\end{proof}

\section{Amitsur's Theorem}

As a consequence of the hook theorem (Theorem \ref {teoprin}), we also have an analogue of Amitsur's Theorem for the case of $PI$-superalgebras with superinvolution or graded involution $\#$, which ensures that any $\#$-$PI$-superalgebra satisfies $\#$-superidentities which are the powers of standard polynomials in any type of variables, where the standard polynomial is the well-known polynomial of the form 
$$St_k(x_1,\dots,x_k)=\sum_{\sigma \in S_k} (-1)^{\sigma}
x_{\sigma(1)} \cdots x_{\sigma(k)}.$$

Let us recall the following simple combinatorial fact. 

\begin{lema} \label{lemta}  
Let $\widetilde{\lambda} = (m^k) \vdash n$ be the rectangle partition of $n=km$ and $T_{\widetilde{\lambda}}$ the standard Young tableau associated to $\widetilde{\lambda}$ of the following form

\begin{table}[H]
\centering
$T_{\widetilde{\lambda}}=$
\begin{tabular}{|c|c|c|c|c}
\cline{1-4}
1 & k+1 & $\cdots$ & $(m-1)k+1$ &  \\ \cline{1-4}
2 & k+2 & $\cdots$ & $(m-2)k+2$ &  \\ \cline{1-4}
 $\vdots$ & $\vdots$    & $\vdots$ &  $\vdots$  &  \\ \cline{1-4}
k & 2k  & $\cdots$  & $mk$    &  \\ \cline{1-4}
\end{tabular}.
\end{table}
\noindent Consider an ordinary multilinear polynomial $f(x_1, x_2, \cdots, x_{km})= e_{T_{\widetilde{\lambda}}}(x_1x_2\cdots x_{km})\in P_{km}$. Then, for the polynomial
$$
w(x_1, x_2, \cdots, x_k)= f(x_1, \cdots, x_k, x_1, \cdots, x_k, \cdots, x_1, \cdots, x_k)
$$
obtained of $f$ identifying variables $x_{k+i}=x_{2k+i}= \cdots= x_{(m-1)k+i}=x_i$ for any $i=1,\cdots, k$, we have
$$
w(x_1, x_2, \cdots, x_k)=(m!)^k St_k^m(x_1, \cdots, x_k).
$$
\end{lema}
\begin{proof}
This lemma is an immediate consequence of the obvious fact that the polynomial $f$ is symmetric in each set of variables
$$ X_1=\{x_1,x_{k+1},\dots,x_{(m-1)k+1}\}, \  \dots, \  X_k=\{x_k,x_{2k},\dots,x_{mk}\}.
$$
\end{proof}

\begin{coro}\label{amitsur1}\textit{(Amitsur's theorem for $\#$-superalgebras)} Let $F$ be a field of characteristic zero and $A$ a superalgebra over $F$ with a superinvolution or a graded involution $\#$. If $A$ is a $PI$-algebra, then there exist integers $d_i,l_i \geq 0$, $i=1,2,3,4$, defined by Theorem \ref{teoprin}, such that $A$ satisfies the $\#$-superidentities 
\begin{eqnarray} \label{AMI}
&&St_{k_1}^{m_1}(y_{0,1},\cdots, y_{0,k_1}) \equiv 0, \qquad
St_{k_2}^{m_2}(z_{0,1},\cdots, z_{0,k_2}) \equiv 0, \nonumber \\
&&St_{k_3}^{m_3}(y_{1,1},\cdots, y_{1,k_3})\equiv 0,  \qquad
St_{k_4}^{m_4}(z_{1,1},\cdots, z_{1,k_4})\equiv 0,
\end{eqnarray}
where $k_i=d_i+1 \geq 1$ and $m_i=l_i+1 \geq 1.$
\end{coro}
\begin{proof}
Let $\chi_{\langle n \rangle}(A)= \displaystyle\sum_{\langle \lambda \rangle\vdash n}m_{\langle \lambda \rangle}\chi_{\langle \lambda \rangle}$ be the $\cn$-th $\#$-cocharacter of $A$. By Theorem \ref{teoprin}, there exist non-negative integers $d_i, l_i\geq 0$, $i=1,2,3,4$, such that $m_{\langle\mu\rangle}=0$ if
$$
D_{\langle\mu\rangle} \not\subseteq \mathcal{H}_{\langle d, l \rangle}\langle n \rangle= (\mathcal{H}(d_1, l_1),\mathcal{H}(d_2, l_2), \mathcal{H}(d_3, l_3), \mathcal{H}(d_4, l_4) ).
$$
Taking $k_i= d_i+1$, $m_i=l_i+1,$ and $\mu(i)=(m_i^{k_i})$ the rectangle partition of $n_i=k_i m_i$, let us consider
the multidiagram $D_{\langle \mu \rangle_i} =(\emptyset,\dots,D_{\mu(i)},\dots,\emptyset)$  of the multipartition $\langle \mu \rangle_i=(0,\dots,\mu(i),\dots,0)$ for any fixed $i=1,2,3,4$, where $D_{\mu(j)}=\emptyset$ is the empty diagram of the zero partition $\mu(j)=(0)$ ($n_j=0$ in $\cn$) for all $j \neq i$.

We have that $D_{\mu(i)} \not \subseteq \mathcal{H}(d_i,l_i)$ and, consequently, $D_{{\langle \mu \rangle}_i} \not \subseteq \mathcal{H}_{\langle d, l \rangle}\langle n \rangle$, what implies that $m_{{\langle \mu \rangle}_i}=0.$
Let us fix the following $\#$-supermonomials
$w_1=y_{0,1}\cdots y_{0, n_1}$, $w_2= z_{0,1}\cdots z_{0, n_2}$,  $w_3=  y_{1,1}\cdots y_{1, n_3}$ and $w_4= z_{1,1}\cdots z_{1, n_4}.$
Consider the standard Young tableau $T_{\mu(i)}$  given by Lemma \ref{lemta} for each $i=1,2,3,4$. Since $m_{\langle \mu \rangle_i}=0$, it follows that
\begin{eqnarray*}
f_i=e_{T_{{\langle \mu \rangle}_i}} w_i=e_{T_{\mu(i)}} w_i \in & Id_2^{\#}(A).
\end{eqnarray*}
Observe that the element $e_{T_{{\langle \mu \rangle}_i}}$ acts only in the corresponding 
set of variables, that is, in the set $\mathcal{Y}_{0,n_1}$ for $i=1,$ $\mathcal{Z}_{0,n_2}$ for $i=2$,  $\mathcal{Y}_{1,n_3}$ for $i=3,$ or $\mathcal{Z}_{1,n_4}$ for $i=4.$ Then, $f_i$ depends on  the same set of variables. 
\noindent Let us denote the following $n_i$-tuples ($n_i=k_i m_i$, for $i=1,2,3,4$) by
$$
\widetilde{X}_1=\mathcal{Y}_{0,k_1}^{m_1}= (y_{0,1},\cdots, y_{0,k_1}, \cdots,y_{0,1},\cdots, y_{0,k_1} ),
$$
repeating the variables of the set $X_1=\mathcal{Y}_{0,k_1}=\{y_{0,1},\cdots,y_{0,k_1}\}$  $m_1$ times,
$$
\widetilde{X}_2=\mathcal{Z}_{0,k_2}^{m_2}=(z_{0,1},\cdots, z_{0,k_2}, \cdots,z_{0,1},\cdots, z_{0,k_2} ),
$$
repeating the variables of the set $X_2=\mathcal{Z}_{0,k_2}=\{z_{0,1},\cdots,z_{0,k_2}\}$  $m_2$ times,
$$
\widetilde{X}_3=\mathcal{Y}_{1,k_3}^{m_3}=(y_{1,1},\cdots, y_{1,k_3}, \cdots,y_{1,1},\cdots, y_{1,k_3} ),
$$
repeating the variables of the set $X_3=\mathcal{Y}_{1,k_3}=\{y_{1,1},\cdots,y_{1,k_3}\}$  $m_3$ times and
$$
\widetilde{X}_4=\mathcal{Z}_{1,k_4}^{m_4}=(z_{1,1},\cdots, z_{1,k_4}, \cdots, z_{1,1},\cdots, z_{1,k_4} ),
$$
repeating the variables of the set $X_4=\mathcal{Z}_{1,k_4}=\{z_{1,1},\cdots,z_{1,k_4}\}$  $m_4$ times. By Lemma \ref{lemta}, one obtains
$$
f_i(\widetilde{X}_i)=f_i(X_i,\dots,X_i)=(m_i!)^{k_i} St_{k_i}^{m_i}(X_i) \in Id_2^\#(A),
$$
as a result of the corresponding identification of variables of $f_i.$
Once $char(F)=0$, it follows that
$$
St_{k_i}^{m_i}(X_i) \in Id_2^\#(A)
$$
for any $i=1,\dots,4,$ and the proof of the corollary is concluded. 
\end{proof}

Notice that the principal value of Corollary \ref{amitsur1} is the correspondence of the pairs $(k_i,m_i)$ in (\ref{AMI}) with the parameters $(d_i,l_i)$ given by Theorem \ref{teoprin} for $i=1,\dots,4.$ The fact that a $\#$-$PI$-superalgebra satisfies 4 identities of the form (\ref{AMI}) for some integers $k_i, m_i \geq 1$ is a simple consequence of classic Amitsur's Theorem for ordinary identities (\cite{Am1}; Theorem 4.5.4 in \cite{RV}). Since any $PI$-algebra satisfies an ordinary identity
$St_k^m(x_1,\dots,x_k) \equiv 0$, then any $\#$-$PI$-superalgebra $A$ satisfies also 4 identities
$St_k^m(\mathcal{Y}_{0,k}) \equiv 0$, \ $St_k^m(\mathcal{Z}_{0,k}) \equiv 0$, \ $St_k^m(\mathcal{Y}_{1,k}) \equiv 0$, \ $St_k^m(\mathcal{Z}_{1,k}) \equiv 0$ for $k_1=k_2=k_3=k_4=k,$
and $m_1=m_2=m_3=m_4=m$ (evaluating $X_k=\{ x_1,\dots,x_k \}=\mathcal{Y}_{0,k},$ \  $X_k=\mathcal{Z}_{0,k},$ \ $X_k=\mathcal{Y}_{1,k},$ and \ $X_k=\mathcal{Z}_{1,k},$ respectively).
But the parameters $(d_i,l_i)$ given by Theorem \ref{teoprin} for any particular $i=1,\dots,4$ can be essentially less than the pair $(k,m)$, defined by the minimal ordinary identity of $A$ of the form $St_k^m(x_1,\dots,x_k) \equiv 0.$  It is an interesting question what is the real correspondence between pairs $(k_i,l_i)$ for $i=1,\dots,4$ in Corollary \ref{amitsur1} and the minimal pairs $(k,m),$ defined by the ordinary identities for a given $\#$-$PI$-superalgebra $A$.

For example, Giambruno, Ioppolo and Martino have found in \cite{GIM} the minimal degrees of the corresponding standard $\#$-superidentities of the full matrix superalgebra $M_{k,k}(F)$ over $F$ of the order $2k$ endowed with the elementary $\mathbb{Z}_2$-grading, defined by the $2k$-tuple $(0^k,1^k),$ and transpose superinvolution $trp$. They have proved that $(M_{k,k}(F),trp)$ satisfies the $\#$-superidentities $St_{2k}(y_{0,1},\dots,y_{0,2k}) \equiv 0,$  \   $St_{2k}(z_{0,1},\dots,z_{0,2k}) \equiv 0,$  \  
$St_{2k}(y_{1,1},\dots,y_{1,2k}) \equiv 0,$  \  
$St_{2k}(z_{1,1},\dots,z_{1,2k}) \equiv 0.$  And these are the standard $\#$-superidentities of the minimal degrees for $(M_{k,k}(F),trp)$, i.e., in terms of Corollary \ref{amitsur1}, we have $k_1=k_2=k_3=k_4=2k,$ $m_i=0$ for all $i=1,\dots,4.$ 

It is worth to mention that the minimal degree of the ordinary standard identity of $M_{k,k}(F)$ is $4k$ (\cite{AL,DR,DrenForm,RV}), and much greater than $2k.$ Also the hook of ordinary identities of $M_{k,k}(F)$ is $\mathcal{H}(4k^2,0)$ ($d= 4 k^2,$ $l=0$) (see \cite{RV,Reg}). Unfortunately, the parameters of the quadruple hook $\mathcal{H}_{(d,l)}$ for $\#$-superidentities of $M_{k,k}(F)$ are unknown, although we can prove that $l_1=l_2=l_3=l_4=0$ also.

\section{Examples}

The following examples exhibit both the decomposition of the $\langle n \rangle$-th  $\#$-cocharacter of the considered algebras and the existence of the quadruple hook provided by Theorem \ref{teoprin} for the 2-generated Grassmann algebra $G$ and the infinitely generated Grassmann algebra $E$ with the canonical $\mathbb{Z}_2$-grading and some superinvolutions.

\begin{exem} \label{GGGE1}\rm Consider $G=\langle e_1, e_2 \ | \ e_1e_2= -e_1e_2, \ e_1^2=0, \ e_2^2=0 \rangle_F=\textrm{span}_{F} \{ e_1,e_2,e_1e_2\}$ the $2$-generated non-unitary Grassmann algebra over a field $F$ of characteristic zero ($\dim G=3$), with the canonical $\mathbb{Z}_2$-grading $G=G_0\oplus G_1$, where $G_0=\textrm{span}_{F} \{ e_1e_2\}$, $G_1=\textrm{span}_{F} \{ e_1,e_2\}$. 

Take the superinvolution $\#$ of $G$ defined by $e_i^{\#}= -e_i,$ $i=1, 2$ \  or $$(\alpha_1 e_1+\alpha_2 e_2+\alpha_3 e_1 e_2)^\#=-\alpha_1 e_1-\alpha_2 e_2+\alpha_3 e_1 e_2, \  \  \alpha_i\in F.$$

$G$ can be decomposed as follows:
$$
G= G_{0}^{+ } \oplus G_{0}^{-} \oplus G_{1}^{+ } \oplus G_{1}^{- },
$$
 where $G_{0}^{+ }= G_{0}$, $G_{0}^{-}= \{0\}$, $G_{1}^{+ }= \{0\}$ and $G_{1}^{- }= G_{1}$. $G$ is also a $PI$-algebra, since it is nilpotent of degree $3$. 

One can directly check that all the $\#$-superidentities of $G$ are consequences of the $\#$-superidentities 
\begin{equation} \label{G-ident}
y_{1,1}\equiv 0, \  z_{0,1} \equiv 0, \  y_{0,1} y_{0,2}\equiv 0,  \  z_{1,1} y_{0,1}\equiv 0, \  y_{0,1} z_{1,1}\equiv 0,  \   z_{1,1} z_{1,2}+z_{1,2} z_{1,1}\equiv 0. 
\end{equation}
In particular, these $\#$-superidentities imply $z_{1,i}z_{1,j} \in (\mathcal{F})_{0}^{+},$ and, thus, the ordinary nilpotency of degree 3 of $G$ is also a consequence of (\ref{G-ident}).
Note that $G$ does not satisfies $y_{0,1} \equiv 0,$ \ $z_{1,1} \equiv 0,$ \ $z_{1,1} z_{1,2} \equiv 0,$ \  $[z_{1,1} z_{1,2}] =  z_{1,1} z_{1,2} -z_{1,1} z_{1,2} \equiv 0.$
The description of $\cn$-th $\#$-cocharacters of $G$ immediately follows from $\#$-superidentities (\ref{G-ident}).
\begin{itemize}
\item $\chi_{\cn}(G)=0, \quad c_{\cn}(G)=0, \quad P_{\cn}\cap Id_2^{\#}(G)= P_{\cn}
\\
\mbox{ for all } \cn=(n_1,n_2,n_3,n_4), \ \mbox{ with } n=n_1+n_2+n_3+n_4 \geq 3 \ \   \mbox{ or } \  n_2+n_3>0;$
\item $\chi_{\cn}(G)=0, \quad c_{\cn}(G)=0, \quad P_{\cn}\cap Id_2^{\#}(G)= P_{\cn} \\
\mbox{ if }  \ \  \cn=(2,0,0,0) \  \  \mbox{ or } \ \ 
\cn=(1,0,0,1);$
\item $\chi_{(1,0,0,0)}(G)= \chi_{(\yng(1),\emptyset,\emptyset,\emptyset)}\subseteq (\mathcal{H}(1,0),\mathcal{H}(0,0),\mathcal{H}(0,0),\mathcal{H}(0,0)), \\
\! \ \  c_{(1,0,0,0)}(G)= 1, \quad P_{(1,0,0,0)}=\text{span}_{F} \{ y_{0,1}\},  \quad P_{(1,0,0,0)} \cap Id_2^{\#}(G)= \{0\};$
\item $\chi_{(0,0,0,1)}(G)= \chi_{(\emptyset,\emptyset,\emptyset,\yng(1))}\subseteq (\mathcal{H}(0,0),\mathcal{H}(0,0),\mathcal{H}(0,0),\mathcal{H}(1,0)), \\
\! \ \  c_{(0,0,0,1)}(G)= 1, \quad P_{(0,0,0,1)}=\text{span}_{F} \{ z_{1,1}\}, \quad P_{(0,0,0,1)}\cap Id_2^{\#}(G)= \{0\};$
\item $\chi_{( 0,0,0,2)}(G)=\chi_{(\emptyset, \emptyset, \emptyset, \yng(1,1))},
\\ 
\! \ \  c_{(0,0,0,2) }(G)=1, \quad P_{ (0,0,0,2)} \cap Id_2^{\#}(G)= \text{span}_{F} \{ z_{1,1}z_{1,2} +z_{1,2}z_{1,1} \}.$
\end{itemize}

We observe that it is possible to choose either
$\chi_{( 0,0,0,2)}(G)\subseteq (\mathcal{H}(0,0),\mathcal{H}(0,0),\mathcal{H}(0,0),\mathcal{H}(0,1)),$ \ or \  $\chi_{( 0,0,0,2)}(G)\subseteq (\mathcal{H}(0,0),\mathcal{H}(0,0),\mathcal{H}(0,0),\mathcal{H}(2,0)).$

In short, we have, for all $\langle n\rangle = (n_1,n_2,n_3,n_4)$ 
$$
\chi_{\cn}(G) = \dis\sum_{\langle\lambda\rangle \vdash \cn \atop \langle\lambda\rangle \in \mathcal{H}_{\langle d, l \rangle}\langle n \rangle} m_{\langle\lambda\rangle}\chi_{\langle\lambda\rangle},
$$
where $\mathcal{H}_{\langle d, l \rangle}\langle n \rangle=(\mathcal{H}(1,0),\mathcal{H}(0,0),\mathcal{H}(0,0),\mathcal{H}(0,1))$, i.e. $\langle d, l \rangle=(1,0;0,0;0,0;0,1),$ or we also can assume $\mathcal{H}_{\langle d, l \rangle}\langle n \rangle=(\mathcal{H}(1,0),\mathcal{H}(0,0),\mathcal{H}(0,0),\mathcal{H}(2,0))$, i.e. $\langle d, l \rangle=(1,0;0,0;0,0;2,0).$
\end{exem}

\begin{exem}
Let $E = \langle 1, e_1, e_2, \cdots \ | \ e_ie_j = -e_je_i, \ \forall \ i,j \geq 1 \rangle_F$ be the infinite-dimensional unitary Grassmann algebra over a field $F$ of characteristic zero with its canonical grading $E = E_0 \oplus E_1$, where $E_0$ is the space generated by monomials in the $e_i$'s of even length and $E_1$ is the space generated by monomials in the $e_i$'s of odd length. Consider $\#:E \to E$ the superinvolution given by $e_i^\# = -e_i$ (see \cite{alj1}). Since $\mathrm{char}(F) = 0$, $E$ can be decomposed as $E = E_0^+ \oplus E_0^- \oplus E_1^+ \oplus E_1^-$, where $E_i^+$ is the space formed by all elements $g_i \in E_i$ such that $g_i^\# = g_i$ for $i = 0, 1$ and $E_j^-$ is formed by elements $g_j \in E_j$ satisfying $g_j^\# = -g_j$ for $j = 0, 1$. Then $E_0^+ = E_0$, $E_0^- = \{0\}$, $E_1^+ = \{0\}$, and $E_1^- = E_1$. Therefore, $E$ satisfies the $\#$-superidentities $z_{0,1} = 0$ and $y_{1,1} = 0$ and, moreover, with this structure, one can adapt the results of \cite{gmz1} to determine the set of other $\#$-superidentities of $E$ and describe the decomposition of the $\#$-cocharacter. Specifically, the results of \cite{gmz1} imply that, assuming $n_2=n_3=0$ and fixing $n_1 + n_4 = n \geq 1$,  we have
$$
\chi_{\cn}(E) = \chi_{((n_1), \emptyset, \emptyset, (1^{n_4}) )} = \chi_{(n_1)} \otimes \chi_{\emptyset} \otimes \chi_{\emptyset} \otimes \chi_{(1^{n_4})}
$$
which is contained in the quadruple hook
$$
(\mathcal{H}(1,0),\mathcal{H}(0,0),\mathcal{H}(0,0),\mathcal{H}(0,1)),
$$
$\chi_{\cn}(E) = 0$ and $c_{\cn}(E) = 0$ for all $\cn = (n_1,n_2,n_3,n_4)$ with $n_2 \geq 1$ or $n_3 \geq 1$, 
and
$$
Id_2^\#(E) = \langle [y_{0,1},y_{0,2}], [y_{0,1},z_{1,2}], z_{1,1}z_{1,2} + z_{1,2}z_{1,1},z_{0,1}, y_{1,1} \rangle_{T_2^{\#}}.
$$
Furthermore, $c_{\cn}(E) = 1$ by the first part of the proof of Proposition 3 in \cite{gmz1}.
\end{exem}

\section{Polynomials of Amitsur for $\#$-Superalgebras}

Let $A$ be a  $\#$-superalgebra. If $A$ is a $PI$-algebra, we know from Theorem \ref{teoprin} that there exist integers $d_i, l_i\geq 0$, with $i= 1,2,3, 4$, such that the $\cn$-th cocharacter of $\#$-superidentities of $A$, $\chi_{\cn}(A)$, is contained in a quadruple hook 
$$
\mathcal{H}_{\langle d,l\rangle}= (\mathcal{H}(d_1, l_1),\mathcal{H}(d_2, l_2), \mathcal{H}(d_3, l_3), \mathcal{H}(d_4, l_4) ),
$$
where $\langle d,l\rangle=(d_1,l_1;d_2,l_2;d_3,l_3;d_4,l_4)$.

Fix some integers $d_i,l_i\geq 0$ for $i=1,2,3,4$. Consider $\langle \hat{\lambda}\rangle=(\hat{\lambda}(1),\hat{\lambda}(2),\hat{\lambda}(3),\hat{\lambda}(4))\vdash \langle k\rangle = (k_1,k_2,k_3,k_4)$, where $\hat{\lambda}(i)=(l_i+1)^{d_i+1}$ is the rectangle partition of $k_i=(l_i+1)(d_i+1)$, $i=1,2,3,4$. Also let us consider the element $\mathcal{E}_{\langle\hat{\lambda}\rangle}=\mathcal{E}_{\hat{\lambda}(1)} \otimes \mathcal{E}_{\hat{\lambda}(2)} \otimes \mathcal{E}_{\hat{\lambda}(3)} \otimes \mathcal{E}_{\hat{\lambda}(4)} \in FS_{\langle k \rangle}$, where
$$\mathcal{E}_{\hat{\lambda}(i)} = \displaystyle \sum_{\sigma_i\in S_{k_i}}\chi_{\hat{\lambda}(i)}(\sigma_i)\sigma_i\in F{S_{k_i}}, \ \ i=1,2,3,4.
$$

Let us recall that the element $\mathcal{E}_{\hat{\lambda}(i)}$ is an essential central idempotent of $FS_{k_i}$, and $I_{\hat{\lambda}(i)}=FS_{k_i}\mathcal{E}_{\hat{\lambda}(i)}$ is the minimal two-sided ideal of $FS_{k_i}$ corresponding to the Young diagram of the shape $\hat{\lambda}(i)$ (see, for instance, \cite{RV,jlieb}),  for $i=1,2,3,4$. 

Then by Proposition \ref{bilat-id}, the element $\mathcal{E}_{\langle\hat{\lambda}\rangle}$ also is the essential central idempotent of $FS_{\langle k \rangle}$ and $FS_{\langle k\rangle} \mathcal{E}_{\langle \hat{\lambda} \rangle}$ is the minimal two-sided ideal of $FS_{\langle k\rangle}$,  which corresponds to the multipartition $\langle \hat{\lambda}\rangle \vdash \langle k \rangle$.

Let us consider the collection of $\#$-superpolynomials of Amitsur of rank \\ 
$\mbox{}$ \quad $\langle d,l\rangle=(d_1,l_1; d_2, l_2;d_3, l_3; d_4,l_4)$ defined by:
\begin{eqnarray} \label{CapSup1}
&&E_{1,(d_1, l_1)}=Am_{(d_1, l_1)}(y_{0,1},\cdots,y_{0,k_1};\bar{x}_{\langle
k_1+1 \rangle}) = \dis\sum_{\sigma\in S_{k_1}} \chi_{\hat{\lambda}(1)}(\sigma) x_1 y_{0,\sigma(1)}x_2\cdots  x_{k_1} y_{0,\sigma(k_1)}x_{k_1+1}, \nonumber
\\
&&E_{2,(d_2, l_2)}=Am_{(d_2, l_2)}(z_{0,1},\cdots,z_{0,k_2};\bar{x}_{\langle k_2+1 \rangle}) = \dis\sum_{\sigma\in S_{k_2}} \chi_{\hat{\lambda}(2)}(\sigma) x_1 z_{0,\sigma(1)}x_2\cdots  x_{k_2} z_{0,\sigma(k_2)}x_{k_2+1}, \nonumber
\\
&&E_{3,(d_3, l_3)}=Am_{(d_3, l_3)}(y_{1,1},\cdots,y_{1,k_3};\bar{x}_{\langle k_3+1 \rangle}) = \dis\sum_{\sigma\in S_{k_3}} \chi_{\hat{\lambda}(3)}(\sigma) x_1 y_{1,\sigma(1)}x_2\cdots  x_{k_3} y_{1,\sigma(k_3)}x_{k_3+1}, \nonumber
\\
&&E_{4,(d_4, l_4)}=Am_{(d_4, l_4)}(z_{1,1},\cdots,z_{1,k_4};
\bar{x}_{\langle k_4+1 \rangle}) = \dis\sum_{\sigma\in S_{k_4}} \chi_{\hat{\lambda}(4)}(\sigma) x_1 z_{1,\sigma(1)}x_2\cdots x_{k_4} z_{1,\sigma(k_4)}x_{k_4+1}, \nonumber  \\
&&\ \ \ \ \ \ \ \mbox{\ \ \  for all } \  x_j \in \{ y_{0,k_1+j}, z_{0,k_2+j}, y_{1,k_3+j}, z_{1,k_4+j}, 1\}, \ j=1,\dots,\widehat{k}+1, \ \widehat{k}=\max_{i=1,\dots,4} k_i,\\
&&\mbox{ } \qquad \quad \bar{x}_{\langle s \rangle}=(x_1,\cdots,x_{s}), \quad s=k_i+1, \ \ i=1, 2, 3, 4, \nonumber
\end{eqnarray}
and \qquad \qquad
$Am_{(d,l)}(a_1, ..., a_k;b_1, ..., b_{k+1})= \dis\sum_{\sigma \in S_{k}}\chi_{\hat{\lambda}}(\sigma)b_1a_{\sigma(1)}b_2...b_ka_{\sigma(k)}b_{k+1}$ \\
is the classical Amitsur polynomial, which corresponds to the partition $\hat{\lambda} = (l+1)^{d+1} \vdash k$ with $k=(l+1)(d+1)$ (see \cite{SAR, RV}). We also call the variables in $Y_{0}$ for the $\#$-superpolynomial $E_{1,(d_1,l_1)}$; the variables in $Z_{0}$ for the $\#$-superpolynomial $E_{2,(d_2,l_2)}$; the variables in $Y_{1}$ for the $\#$-superpolynomial $E_{3,(d_3,l_3)}$; and the variables in $Z_{1}$ for the $\#$-superpolynomial $E_{4,(d_4,l_4)}$ by special variables.

\begin{defi} \label{deficapelli1}
Let $A$ be a superalgebra with a superinvolution or a graded involution $\#$. We say that $A$ satisfies Amitsur $\#$-superidentity of rank $\langle d,l\rangle=(d_1,l_1; d_2, l_2;d_3, l_3; d_4,l_4)$, where $d_i,l_i\geq 0$, $i=1,2,3,4$, if $A$ satisfies the collection of all $\#$-superidentities of the form
$E_{i,(d_i,l_i)} \equiv 0$ for $i=1,2,3,4$, as defined in (\ref{CapSup1}), and the variables $x_j$ can be replaced in all possible ways by the elements of the set $\{ y_{0,k_1+j}, z_{0,k_2+j}, y_{1,k_3+j}, z_{1,k_4+j}, 1\}, \ j=1,\dots,\widehat{k}+1, \ \widehat{k}=\max_{i=1,\dots,4} k_i$ (we assume that the variable $x_j$ is omitted if the choice $x_j=1$ is made).

In this case, we say that $A$ satisfies $E_{\langle d,l\rangle}$ ($E_{\langle d,l\rangle}\subseteq Id_2^\#(A)$).
\end{defi}

Observe that $E_{\langle d,l\rangle}$ are analogues of the classic Amitsur polynomial $Am_{(d,l)}$ defined in \cite{SAR} (see also Theorem 4.7.2 in \cite{RV}) for $\#$-superidentities. Note also that in the classic case of ordinary identities, any Capelli polynomial $Cap_{k}$ is the partial case of the Amitsur polynomial $Am_{(k,0)}$ when $l=0.$ 

The following theorem presents the translation of the hook theorem (Theorem \ref{teoprin}) on the language of $\#$-superpolynomials for the case of $\#$-superidentities and generalizes the classic theorem about Amitsur polynomials (\cite{SAR}) in a natural way for $\#$-superidentities.

\begin{teo} \label{AmPol}
Let $A$ be a superalgebra with a superinvolution or a graded involution $\#$, 
$
\chi_{\cn}(A) = \dis\sum_{\langle\lambda\rangle\vdash \cn}m_{\langle\lambda\rangle}\chi_{\langle\lambda\rangle}
$
its $\cn$-th $\#$-cocharacter. If $A$ is a $PI$-algebra, then $A$ satisfies $E_{\langle d,l\rangle}$ (the Amitsur $\#$-superidentity of rank $\langle d,l\rangle$) for some $\langle d,l\rangle=(d_1,l_1;d_2,l_2;d_3,l_3;d_4,l_4)$ if, and only if, $m_{\langle\lambda\rangle} = 0$ whenever $\langle\lambda\rangle \notin \mathcal{H}_{\langle d,l\rangle}$.
\end{teo}

\begin{proof}
Firstly, suppose that $A$ satisfies $E_{\langle d,l\rangle}$ (the Amitsur $\#$-superidentity of rank $\langle d,l\rangle$) for some  $\langle d,l\rangle$, with $d_i,l_i\geq 0$, for $i=1,2,3,4$. By contradiction, assume that there exists a multipartition \((\lambda(1),\lambda(2),\lambda(3),\lambda(4))=\langle \lambda \rangle \vdash \langle n\rangle\), where \(n = n_1 + n_2 + n_3 + n_4\) with \(n_i \geq 0\) for \(i = 1, 2, 3, 4\), such that \(m_{\langle \lambda \rangle} \neq 0\) and \(\langle \lambda \rangle \notin \mathcal{H}_{\langle d,l\rangle}\). By Proposition \ref{propetm}, there exist a multilinear $\#$-superpolynomial $f=f(\mathcal{Y}_{0,n_1},\mathcal{Z}_{0,n_2},\mathcal{Y}_{1,n_3},\mathcal{Z}_{1,n_4})\in P_{\cn}$ and a multitableau $T_{\langle \lambda \rangle}$ of the shape $\langle\lambda\rangle$ such that $e_{T_{\langle\lambda\rangle}}f\notin Id_2^\#(A)$. Also, there exists $i\in\{1,2,3,4\}$ such that $\lambda(i) \vdash n_i$ satisfies $\lambda(i) \notin \mathcal{H}(d_i,l_i)$, where $n_i\geq (d_i+1)(l_i+1)$. Without loss of generality, we may assume that $i=1$. Therefore, 
\[
e_{T_{\lambda(1)}}\cdot w_1 y_{0,i_1}w_2 \cdots y_{0,i_{n_1}}w_{n_1+1} \notin Id_2^{\#}(A), \]
where the essential idempotent $e_{T_{\lambda(1)}} \in FS_{n_1}$ acts in the variables $\mathcal{Y}_{0, n_1}$, the indices $i_1, ..., i_{n_1}$ are the entries of the Young tableau $T_{\lambda(1)}$ and $w_j$'s are the monomials, possibly empty, in variables $X\setminus \mathcal{Y}_{0}$ such that $w_1y_{0,i_1}w_2 \cdots y_{0,i_{n_1}}w_{n_1+1}$ is a multilinear $\#$-supermonomial, obtained following the same arguments contained in the proof of Proposition \ref{prop_w_mon}, choosing corresponding monomials of considered $\#$-superpolynomials, considering the left $FS_{\cn}$-module structure of $Id_2^{\#}(A),$ using the structure and applying the properties of the element $e_{T_{\langle\lambda\rangle}} \in FS_{\cn}$, separating the multiplier 
$e_{T_{\lambda(1)}},$ and regrouping indeterminates.

Since $\lambda(1) \notin \mathcal{H}(d_1,l_1)$, $\lambda(1)$ contains the rectangle $\hat{\lambda}(1)=(l_1+1)^{d_1+1} \vdash (l_1+1)(d_1+1)=k_1$ and let $T_{\hat{\lambda}(1)}$ be the corresponding subtableau of $T_{\lambda(1)}$. This means that $D_{\lambda(1)}\supseteq D_{\hat{\lambda}(1)}$. After renaming the variables, we may clearly assume that $T_{\hat{\lambda}(1)}$ is filled up with the integers $1, \dots, k_1$.

Then we decompose $C_{T_{\lambda(1)}}\subseteq S_{n_1}$ into the union of the right cosets of its subgroup $C_{T_{\hat{\lambda}(1)}}$, and $R_{T_{\lambda(1)}}\subseteq S_{n_1}$ into the union of the left cosets of $R_{T_{\hat{\lambda}(1)}}.$ It follows from Lemma \ref{Lem 2} that $e_{T_{\lambda(1)}} = \alpha e_{T_{\hat{\lambda}(1)}}\beta$ for some $\alpha,\beta \in FS_{n_1}$ and, consequently, $e_{T_{\lambda(1)}}\cdot w_1 y_{0,i_1}w_2 \cdots y_{0,i_{n_1}}w_{n_1+1}$ is a linear combination of $\#$-superpolynomials of the type
$$
 \alpha \  \sum_{\sigma \in R_{T_{\hat{\lambda}(1)}} \atop \tau \in C_{T_{\hat{\lambda}(1)}}} (\operatorname{sgn} \tau) \hat{w}_1 y_{0,\sigma\tau\psi(1)} \hat{w}_2 \cdots \hat{w}_{k_1} y_{0,\sigma\tau\psi(k_1)} \hat{w}_{k_1+1},
$$
where $\hat{w}_1, \dots, \hat{w}_{k_1+1}$ are (eventually trivial) $\#$-supermonomials in the remaining variables in the set $(X\setminus \mathcal{Y}_{0,n_1})\cup\{y_{0,\psi(k_1+1)}, \dots, y_{0,\psi(n_1)}\},$ \  $\psi\in S_{n_1}$ and the tableau $T_{\hat{\lambda}_1}$ is filled up with the integers $\{ \psi(1), \dots, \psi(k_1) \}.$ Moreover, at least one of the $\#$-superpolynomials 
$$
\sum_{\sigma \in R_{T_{\hat{\lambda}(1)}} \atop \tau \in C_{T_{\hat{\lambda}(1)}}} (\operatorname{sgn} \tau) \hat{w}_1 y_{0,\sigma\tau\psi(1)} \hat{w}_2 \cdots \hat{w}_{k_1} y_{0,\sigma\tau\psi(k_1)} \hat{w}_{k_1+1},
$$
is not a $\#$-superidentity of $A$. By multiplying on the left by $\psi^{-1}$, we deduce that
\begin{eqnarray*}
&&\psi^{-1} \sum_{\sigma \in R_{T_{\hat{\lambda}(1)}} \atop \tau \in C_{T_{\hat{\lambda}(1)}}} (\operatorname{sgn} \tau) \hat{w}_1 y_{0,\sigma\tau\psi(1)} \hat{w}_2 \cdots \hat{w}_{k_1} y_{0,\sigma\tau\psi(k_1)} \hat{w}_{k_1+1}
\end{eqnarray*}
\begin{eqnarray} \label{betar1}
&&= (\psi^{-1} e_{T_{\hat{\lambda}(1)}} \psi)(\bar{w}_1 y_{0,1} \bar{w}_2 \cdots \bar{w}_{k_1} y_{0,k_1} \bar{w}_{k_1+1}) \notin Id_2^{\#}(A),  
\end{eqnarray}
because $Id_2^{\#}(A)$ is a left $FS_{\langle n\rangle}$-module and  particularly, is a left $FS_{n_1}$-module, where $\bar{w}_j$ in (\ref{betar1}) are multilinear $\#$-supermonomials, obtained by renaming (by $\psi^{-1}$) the variables $\{y_{0,\psi(k_1+1)}, \dots, y_{0,\psi(n_1)}\}$ of $\hat{w}_j$ into the variables $\{y_{0,k_1+1}, \dots, y_{0,n_1}\},$ respectively.

Recalling that the element $\mathcal{E}_{\hat{\lambda}(1)} \in FS_{k_1}$ generates $I_{\hat{\lambda}(1)}$ as a left ideal (see, for example, Proposition 2.2.2, p. 47, in \citep{RV}). Observe that the tableau $\psi^{-1} T_{\hat{\lambda}(1)}$ is filled up with the integers $\{ 1, \dots, k_1 \}.$ Hence, we have  $\psi^{-1} e_{T_{\hat{\lambda}(1)}} \psi = e_{\psi^{-1} T_{\hat{\lambda}(1)}} = b \in FS_{k_1}.$ Since $b$ is an essential idempotent of $FS_{k_1}$ corresponding to $\hat{\lambda}(1)$, it follows that $b \in I_{\hat{\lambda}(1)}$. Therefore, there exists an element $a = \sum_{\rho \in S_{k_1}} \gamma_\rho \rho \in FS_{k_1}$ such that $a \mathcal{E}_{\hat{\lambda}(1)} = b$. Hence,
(\ref{betar1}) implies
\begin{eqnarray*}
&&b \cdot (\bar{w}_1 y_{0,1} \bar{w}_2 \cdots \bar{w}_{k_1} y_{0,k_1} \bar{w}_{k_1+1}) = a \  \mathcal{E}_{\hat{\lambda}(1)}  \cdot (\bar{w}_1 y_{0,1} \bar{w}_2 \cdots \bar{w}_{k_1} y_{0,k_1} \bar{w}_{k_1+1}) \notin Id_2^{\#}(A). 
\end{eqnarray*}
Thus, we have 
\begin{eqnarray}\label{betar2}
&&\mathcal{E}_{\hat{\lambda}(1)}  \cdot (\bar{w}_1 y_{0,1} \bar{w}_2 \cdots \bar{w}_{k_1} y_{0,k_1} \bar{w}_{k_1+1}) \notin Id_2^{\#}(A),
\end{eqnarray}
since $Id_2^{\#}(A)$ is a left $FS_{n_1}$-module and $k_1 \leq n_1$. Note that the element $\mathcal{E}_{\hat{\lambda}(1)} \in FS_{k_1}$ acts on variables
$y_{0,1}, \dots, y_{0,k_1}$. 

Once $\bar{w}_j's$ are multilinear $\#$-supermonomials in the variables belonging to the set $(X\setminus \mathcal{Y}_{0,n_1})\cup\{y_{0,k_1+1}, \dots, y_{0,n_1}\},$ we have that $\bar{w}_j's$ are $\mathbb{Z}_2$-homogeneous, and
denote $\deg \bar{w}_j=\gamma_j \in \{ 0, 1 \}.$
When we evaluate the variables of a $\#$-supermonomial $\bar{w}_j$ in the elements of $A_0^+\cup A_0^-\cup A_1^+\cup A_1^-,$ we obtain an evaluation $\widetilde{w}_j$ of $\bar{w}_j$ such that $\widetilde{w}_j$ is also $\mathbb{Z}_2$-homogeneous element with $deg \ \widetilde{w}_j = deg \ \bar{w}_j = \gamma_j$, for all $j=1,2,\cdots, k_1+1$. Besides that, we may decompose $\widetilde{w}_j$ in $\widetilde{w}_j=\tilde{u}_j^{\gamma_j} + \tilde{v}_j^{\gamma_j}$, where $\tilde{u}_j^{\gamma_j}\in A_{\gamma_j}^+$ and $\tilde{v}_j^{\gamma_j}\in A_{\gamma_j}^-$ are the symmetric and skew parts of $\widetilde{w}_j$, respectively. Replacing in (\ref{betar2}) the nonempty $\#$-supermonomials $\bar{w}_j$ by $y_{{\gamma_j},k_1+j}+z_{{\gamma_j},k_1+j}$, for $j=1,2,\cdots,k_1+1$, we obtain a new $\#$-superpolynomial of the form
\begin{eqnarray} \label{bet3}
&&\mbox{} \quad  \mathcal{E}_{\hat{\lambda}(1)} \left[(y_{{\gamma_1},k_1+1}+z_{{\gamma_1},k_1+1}) y_{0,1} (y_{{\gamma_2},k_1+2}+z_{{\gamma_2},k_1+2}) y_{0,2} \cdots y_{0,k_1} (y_{{\gamma_{k_1+1}}, 2k_1+1}+z_{{\gamma_{k_1+1}}, 2k_1+1})\right].
\end{eqnarray}
Notice that some sums of variables $y_{{\gamma_j},k_1+j}+z_{{\gamma_j},k_1+j}$ can be absent in this product, if the corresponding $\#$-supermonomial $\bar{w}_j$ in (\ref{betar2}) is empty. It is clear that the $\#$-superpolynomial (\ref{betar2}) is obtained from (\ref{bet3}) by the replacement $y_{{\gamma_j},k_1+j}=\bar{u}_j^{\gamma_j}$ and $z_{{\gamma_j},k_1+j}=\bar{v}_j^{\gamma_j}$ in $\mathcal{F},$ where $\bar{u}_j^{\gamma_j}$ and $\bar{v}_j^{\gamma_j}$ are the symmetric and skew parts of $\bar{w}_j$ respectively (see more details in the proof of Proposition \ref{prop_w_mon}). This replacement is a graded endomorphism of the free $\#$-superalgebra $\mathcal{F}$ commuting with $\#$. Hence, (\ref{bet3}) is also not a $\#$-superidentity of $A.$ Therefore, we may conclude that, for some variables $x_j\in\{ y_{0,k_1+j}, z_{0,k_1+j}, y_{1,k_1+j}, z_{1,k_1+j}, 1\}$ for $j=1,2,\cdots,k_1+1$, one has
$$
\mathcal{E}_{\hat{\lambda}(1)} (x_1 y_{0,1} x_2 \cdots x_{k_1} y_{0,k_1} x_{k_1+1}) \not \equiv 0
$$
in $A$. Since $\mathcal{E}_{\hat{\lambda}(1)} = \sum_{\sigma \in S_{k_1}} \chi_{\hat{\lambda}(1)}(\sigma) \sigma$ and 
$$
\mathcal{E}_{\hat{\lambda}(1)}(x_1 y_{0,1} x_2 \cdots x_{k_1} y_{0,k_1} x_{k_1+1}) = E_{1,(d_1,l_1)}(y_{0,1}, \dots, y_{0,k_1}; x_1, \dots, x_{k_1+1}),
$$
we obtain
\[
E_{1,(d_1,l_1)}(y_{0,1}, \dots, y_{0,k_1}; x_1, \dots, x_{k_1+1}) = \sum_{\sigma \in S_{k_1}} \chi_{\hat{\lambda}(1)}(\sigma)x_1 y_{0,\sigma(1)} \cdots y_{0,\sigma(k_1)}x_{k_1+1} \notin Id_2^{\#}(A).
\]
This contradicts the hypothesis and proves the first part of the result.

Conversely, suppose that $m_{\langle\lambda\rangle} = 0$ in $\chi_{\cn}(A)$ whenever $\langle\lambda\rangle \notin \mathcal{H}_{\langle d,l\rangle}$. Consider $\langle\hat{\lambda}\rangle = (\hat{\lambda}(1),\hat{\lambda}(2),\hat{\lambda}(3),\hat{\lambda}(4))\vdash \langle k\rangle = (k_1,k_2,k_3,k_4)$, where $\hat{\lambda}(i)=((l_i+1)^{d_i+1}) \vdash k_i=(l_i+1)(d_i+1)$ for $i=1,2,3,4$. Let us fix any $i=1,2,3,4$ and consider for this $i$ the $\#$-superpolynomial $E_{i,(d_i,l_i)}$ for any possible evaluation of variables $x_1, \dots, x_{k_i +1}$ defined in (\ref{CapSup1}). Let $\langle t\rangle_i=(t_{i,1},t_{i,2},t_{i,3},t_{i,4})$ be the multidegree of $E_{i,(d_i,l_i)}$, where $t_{i,1}=deg_{Y_0} E_{i,(d_i,l_i)},$ $t_{i,2}=deg_{Z_0}  E_{i,(d_i,l_i)},$ $t_{i,3}=deg_{Y_1}  E_{i,(d_i,l_i)},$ $t_{i,4}=deg_{Z_1}  E_{i,(d_i,l_i)},$ and the degree of $E_{i,(d_i,l_i)}$ in its special variables is $t_{i,i}\geq k_i$. We have that the space $P_{\langle t \rangle_i}$ of all multilinear $\#$-superpolynomials of multidegree $\langle t \rangle_i$ is a left $F S_{\langle t \rangle_i}$-module, and particularly, $P_{\langle t\rangle_i}$ is a left $F S_{t_{i,i}}$-module if we consider only the action in the special variables of $E_{i,(d_i,l_i)}$.

It is well known (see, for example, \cite{RV,jkerb}), that $FS_{k_i} \mathcal{E}_{\hat{\lambda}(i)}=I_{\hat{\lambda}(i)}$
is a minimal two-sided ideal of $FS_{k_i}$ with the character $d_{\hat{\lambda}(i)} \chi_{\hat{\lambda}(i)}$, where 
$\mathcal{E}_{\hat{\lambda}(i)} = \displaystyle \sum_{\sigma_i\in S_{k_i}}\chi_{\hat{\lambda}(i)}(\sigma_i)\sigma_i\in F{S_{k_i}}.$
Recall that
$$E_{i,(d_i,l_i)}(a_{\gamma,1}, \dots, a_{\gamma,k_i}; x_1, \dots, x_{k_i+1})=\mathcal{E}_{\hat{\lambda}(i)}(x_1 a_{\gamma,1} x_2 \cdots x_{k_i} a_{\gamma,k_i} x_{k_i+1})$$
for any possible choice of $x_j,$ where $a_{\gamma, j}=y_{\gamma, j}$ if $i=1, 3$ or $a_{\gamma, j}=z_{\gamma, j}$ if $i=2, 4,$ and $\gamma=0$ if $i=1, 2$ or 
$\gamma=1$ if $i=3, 4.$
Then $M_i=FS_{t_{i,i}} E_{i,(d_i,l_i)}=FS_{t_{i,i}} \mathcal{E}_{\hat{\lambda}(i)} w$ is an $FS_{t_{i,i}}$-submodule of $P_{\langle t\rangle_i}$, where $FS_{t_{i,i}}$ acts only in the special variables of $E_{i,(d_i,l_i)}$, and $w$ is some multilinear $\#$-supermonomial of degree $t_{i,i} \geq k_i$ in the special variables. By the branching rule of the symmetric group, $FS_{t_{i,i}}$-character of $M_i$ is 
\[
\widetilde{\chi}(i)=\chi(FS_{t_{i,i}} E_{i,(d_i,l_i)}) = \sum_{\mu_i \vdash t_{i,i} \atop \mu_i \geq \hat{\lambda}(i)} \widetilde{m}_{\mu_i} \chi_{\mu_i}
\]
for some multiplicities $\widetilde{m}_{\mu_i}.$ Then by the natural properties of the tensor product of representations of groups the $\#$-character of the left $FS_{\langle t\rangle_i}$-module $FS_{\langle t\rangle_i} E_{i,(d_i,l_i)}$ is equal to
\begin{eqnarray*}
&&\widetilde{\chi}(1) \otimes \cdots \otimes \widetilde{\chi}(4) =
\chi(FS_{t_{i,1}} E_{i,(d_i,l_i)}) \otimes \cdots \otimes \chi(FS_{t_{i,4}} E_{i,(d_i,l_i)})=\\
&&\left( \sum_{\theta_1 \vdash t_{i,1}} \widetilde{m}_{\theta_1} \chi_{\theta_1} \right)
\otimes \cdots \otimes
\left(\sum_{\mu_i \vdash t_{i,i} \atop \mu_i \geq \hat{\lambda}(i)} \widetilde{m}_{\mu_i} \chi_{\mu_i} \right)
\otimes \cdots \otimes
\left( \sum_{\theta_4 \vdash t_{i,4}} \widetilde{m}_{\theta_4} \chi_{\theta_4} \right)
=\\
&&\sum_{ j \neq i} \sum_{\theta_j \vdash t_{i,j}} \sum_{\mu_i \vdash t_{i,i} \atop \mu_i \geq \hat{\lambda}(i)} 
(\widetilde{m}_{\theta_1} \cdots \widetilde{m}_{\mu_i} \cdots \widetilde{m}_{\theta_4}) \  \chi_{(\theta_1,\dots,\mu_i,\dots,\theta_4)},
\end{eqnarray*}
where $\chi_{(\theta_1,\dots,\mu_i,\dots,\theta_4)}=
\chi_{\theta_1} \otimes \cdots \otimes
\chi_{\mu_i} \otimes \cdots \otimes  \chi_{\theta_4}.$ 

Since the diagrams $D_{\mu_i}$ contain the diagram $D_{\hat{\lambda}(i)}$ for each partition $\mu_i\vdash t_{i,i}$ such that $\mu_i\geq\hat{\lambda}(i)$, this implies that the multipartition $\langle\hat{\mu}\rangle=(\theta_1,\ldots,\mu_i,\ldots,\theta_4)\notin \mathcal{H}(d_1,l_1;d_2,l_2;d_3,l_3;d_4,l_4)$ for all $\theta_j \vdash t_{i,j}$, $j\in \{1,2,3,4\}\setminus\{ i\}$. 
By Lemma \ref{Met1}, we have that 
\begin{eqnarray*}
FS_{\langle t\rangle_i} E_{i,(d_i,l_i)} \subseteq 
\bigoplus_{{\langle \hat{\mu} \rangle \vdash \langle t \rangle_i} \atop \hat{\mu} \notin \mathcal{H}_{\langle d, l \rangle}} \left(
\sum_{T_{\langle \hat{\mu} \rangle} \in \mathcal{ST_{\langle \hat{\mu} \rangle}}}
\sum_{f \in P_{\langle t\rangle_i}} \ 
FS_{\langle t\rangle_i} e_{T_{\langle \hat{\mu} \rangle}} f  \right).
\end{eqnarray*}
By hypothesis, $m_{\langle\hat{\mu}\rangle}=0$ in the $\langle t\rangle_i$-th $\#$-cocharacter of $A$ for all $\hat{\mu} \notin \mathcal{H}_{\langle d, l \rangle}.$
It follows from Proposition \ref{propetm} that
$e_{T_{\langle \hat{\mu} \rangle}} f \in Id_2^\#(A)$, for all
$f \in P_{\langle t\rangle_i,}$ standard multitableaux 
$T_{\langle \hat{\mu} \rangle}$ of the shape $\hat{\mu} \vdash \langle 
t \rangle_i,$ such that $\hat{\mu} \notin \mathcal{H}_{\langle d, l \rangle}.$ Thus, $FS_{\langle t\rangle_i} E_{i,(d_i,l_i)} \subseteq P_{\langle t\rangle_i} \cap \operatorname{Id}_2^\#(A)$ and, consequently, $E_{i,(d_i,l_i)} \equiv 0$ is a $\#$-superidentity of $A$ for any $i=1,2,3,4$ and for all possible evaluations of variables $x_1, \cdots, x_{k_i +1}$ defined in (\ref{CapSup1}).
\end{proof}

Theorem \ref{teoprin} jointly with Theorem \ref{AmPol} immediately implies the following corollary.

\begin{coro}
Any $\#$-PI-superalgebra over a field of characteristic zero
satisfies an Amitsur $\#$-superidentity $E_{\langle d,l\rangle}$  of some rank $\langle d,l\rangle=(d_1,l_1;d_1,l_2;d_3,l_3;d_4,l_4).$ 
\end{coro}


\begin{thebibliography}{9}

\bibitem{alj1}
Aljadeff E., Giambruno A., Karasik Y.,
{\it Polynomial Identities with Involution, Superinvolutions and the Grassmann Envelope}, Proc. Amer. Math. Soc., \textbf{145} (2017), 
1843--1857.

\bibitem{AljBelov}
Aljadeff E., Kanel-Belov A., 
{\it Representability and Specht Problem for G-graded Algebras}, Advances in Mathematics, \textbf{225} (2010), 2391--2428.

\bibitem{Am1}
Amitsur S.A., 
{\it The identities of PI-rings}, Proc. Amer. Math. Soc.,
\textbf{4} (1953), 27--34.  

\bibitem{AL}
 Amitsur S.A., Levitzki J., 
{\it Minimal identities for algebras}, 
Proc. Amer. Math. Soc., \textbf{1} (1950), 449--463.

\bibitem{SAR}
Amitsur S. A, Regev A.,
{\it P.I. algebras and their cocharacters}, 
Journal of Algebra, \textbf{78} (1982), 248--254.

\bibitem{A.B.1}
Berele A., 
{\it Homogeneous Polynomial Identities}, 
Israel Journal of Mathematics, \textbf{42} (1982), 258--272.

\bibitem{BH}
Boerner H., 
{\it Representations of groups}, 2nd ed., North-Holland, Amsterdan, 1970.

\bibitem{CG1}
Cirrito A., Giambruno A., 
{\it Group graded algebras and multiplicities bounded by a constant}, Journal of Pure and Applied Algebra, \textbf{217} (2013), 259--268.

\bibitem{CR} 
Curtis  C. W., Reiner I., 
{\it Representation theory of finite groups and associative algebras}, Interscience, John Wiley and Sons, New York-London, 1962.

\bibitem{DR}
Drensky V., 
{\it Free algebras and PI-algebras}, Graduate course in algebra, Springer-Verlag Singapore, Singapore, 2000.

\bibitem {DrenForm} 
Drensky V., Formanek E., 
{\it Polynomial identity rings}, Springer Basel AG, 2004.

\bibitem{fulton} 
Fulton W., 
{\it Young Tableaux: with applications to representation theory and geometry}, Cambridge University Press, 1997.

\bibitem{GID}
Giambruno, A.; Ioppolo, A.; La Mattina, D.,
{\it Varieties of algebras with superinvolution of almost polynomial growth}, Algebr. Represent. Theory,  \textbf{19} (2016), 599--611.  

\bibitem{GIM} 
Giambruno A., Ioppolo A., Martino F.,
{\it Standard polynomials and matrices with superinvolutions},
Linear Algebra and its Applications, \textbf{504} (2016), 272–-291.

\bibitem{gmv1}
Giambruno A., Milies C. P., Valenti A., 
{\it Star-polynomial identities: Computing the exponential growth of the codimensions}, Journal of Algebra, \textbf{469} (2017), 
302--322.

\bibitem{gmz1}
Giambruno A., Mishchenko S., Zaicev M.,
{\it Polynomial Identities on Superalgebras and almost Polynomial Growth}, Communications in Algebra, \textbf{29} (2001), 
3787--3800.

\bibitem{RG}
Giambruno A.,  Regev, A.,
{\it Wreath Products and PI-Álgebras},
Journal of Pure and Applied Algebra, \textbf{35} (1985), 133--149. 

\bibitem{ARA}
Giambruno A., Santos R. B., Vieira, A.,
{\it Identities of $*$-Superalgebras and Almost Polynomial Growth},
Linear and Multilinear Algebra, \textbf{64} (2016), 484--581.

\bibitem{RV}
Giambruno A.,  Zaicev M., 
{\it Polynomial Identities and Asymptotic Methods}, Math. Survey Monogr., vol. 122, Amer. Math. Soc., Providence, RI, 2005.

\bibitem{ioio3}
Ioppolo A., 
{\it The exponent for superalgebras with superinvolution},
Linear Algebra and its Applications, \textbf{555} (2018), 1--20.

\bibitem{IOIO}
Ioppolo A., Martino F., 
{\it On multiplicities of cocharacters for algebras with superinvolution}, Journal of Pure and Applied Algebra, \textbf{225} (2021), n.3, 106536.

\bibitem{jacobs}
Jacobson N., 
{\it Structure of Rings}, 
Colloquium Publications, vol. 37, Amer. Math. Soc., 1956.

\bibitem{jkerb}
James G., Kerber A., 
{\it The Representation Theory of the Symmetric Group},
Addison-Wesley Publishing Company, London, 1981.

\bibitem{jlieb}
James G., Lieback M., 
{\it Representations and Characters of Groups}, 
Cambridge, 2001.

\bibitem{Kem1}
Kemer A.R., 
{\it Ideals of identities of associative algebras}, 
Amer.Math.Soc. Translations of Math. Monographs 87, Providence, R.I., 1991.

\bibitem{AR}
Regev A., 
{\it Existence of Identities in $A \otimes B$}, 
Israel Journal of Mathematics, \textbf{11} (1972), 131--152. 

\bibitem{Reg}
Regev A., 
{\it Algebras satisfying a Capelli identity},
Israel Journal of Mathematics, \textbf{33} (1979), 149--154.

\bibitem{sagan}
Sagan B. S., 
{\it The Symmetric Group: Representations, Combinatorial Algorithms, and Symmetric Functions}, 
2nd ed., Springer, USA, 2000. 

\bibitem{ben1}
Steinberg B., 
{\it Representation Theory of Finite Groups: an Introductory Approach}, 
Springer, USA, 2012. 

\bibitem{IRINA}
Sviridova I., 
{\it Identities of $PI$-algebras Graded by a Finite Abelian Group}, Communications in Algebra, \textbf{39} (2011), 3462--3490.

\bibitem{AV2}
Vieira A. C., 
{\it Finitely generated algebras with involution and multiplicities bounded by a constant}, 
Journal of Algebra, \textbf{422} (2015), 487--503.

\end{thebibliography}
\end{document}